\documentclass[12pt,a4paper]{article}
\usepackage{mathrsfs}
\usepackage{epsfig, graphicx}
\usepackage{latexsym,amsfonts,amsbsy,amssymb}
\usepackage{amsmath,amsthm}
\usepackage{color}
\usepackage{hyperref,color}
\makeatletter 

\@addtoreset{equation}{section}
\makeatother 
\allowdisplaybreaks 

\newtheorem{theorem}{Theorem}[section]
\newtheorem{lemma}{Lemma}[section]
\newtheorem{corollary}{Corollary}[section]
\newtheorem{remark}{Remark}[section]

\begin{document}
\title{Fully Computable Error Bounds for Eigenvalue Problem\footnote{The work
of Hehu Xie is supported in part by the National Natural Science Foundations
of China (NSFC 91330202, 11371026, 11001259, 11031006, 2011CB309703),
the National Center for Mathematics and Interdisciplinary Science
the national Center for Mathematics and Interdisciplinary Science, CAS.}}
\author{Hehu Xie\footnote{LSEC, ICMSEC, Academy of Mathematics and Systems Science,
Chinese Academy of Sciences, Beijing 100190, P.R. China (hhxie@lsec.cc.ac.cn)},\ \
Meiling Yue\footnote{LSEC, ICMSEC, Academy of Mathematics and Systems Science,
Chinese Academy of Sciences, Beijing 100190, P.R. China (yuemeiling@lsec.cc.ac.cn)}\ \ \
and \  Ning Zhang\footnote{LSEC, ICMSEC, Academy of Mathematics and Systems Science,
Chinese Academy of Sciences, Beijing 100190, P.R. China (zhangning114@lsec.cc.ac.cn)}}
\date{}
\maketitle
\begin{abstract}
This paper is concerned with the computable error estimates for the eigenvalue
problem which is solved by the general conforming finite element methods on the
general meshes. Based on the computable error estimate, we can give
an asymptotically lower bound of the general eigenvalues. Furthermore, we also give a guaranteed
upper bound of the error estimates for the first eigenfunction approximation and
a guaranteed lower bound of the first eigenvalue based on
computable error estimator. Some numerical examples are presented to validate the theoretical results
deduced in this paper.

\vskip0.3cm {\bf Keywords.} Eigenvalue problem, computable error estimate,
guaranteed upper bound, guaranteed lower bound, complementary method.

\vskip0.2cm {\bf AMS subject classifications.} 65N30, 65N25, 65L15, 65B99.
\end{abstract}

\section{Introduction}
This paper is concerned with the computable error estimates for the eigenvalue problem by the finite element
method. As we know, the priori error estimates can only give the asymptotic convergence order. The a posteriori
error estimates are very important for the mesh adaption process. About the a posteriori error estimate for the
partial differential equations by the finite element method, please refer to
\cite{AinsworthOden,BabuskaRheinboldt_1,BabuskaRheinboldt_2,BrennerScott,NeittaanmakiRepin,Repin,Verfurth}
 and the references cited therein.

It is well known that the numerical approximations by the conforming finite element methods are upper
bounds of the exact eigenvalues. Recently, how to obtain the lower bounds of the desired eigenvalues
is a hot topic since it has many applications
in some classical problems \cite{ArmentanoDuran,CarstensenGallistl,CarstensenGedicke,HuHuangLin,
LinLuoXie_lowerbound,LinXie_lowerbound,LinXieLuoLiYang,Liu,LiuOishi,SebestovaVejchodsky,YangZhangLin,ZhangYangChen}.
So far, there have developed the nonconforming finite element methods, interpolation
constant based methods and computational error estimate methods. The nonconforming finite element methods can only
obtain the asymptotically lower bounds with the lowest order accuracy.
The interpolation constant method can only obtain the efficient lowest order accuracy on the quasi-uniform meshes.
The interesting computational error method need a condition that the numerical approximation is closer to
the first eigenvalue than the second one. But the paper \cite{SebestovaVejchodsky} gives a clue to us.

This paper is to give computable error estimates for the eigenpair approximations. We produce a guaranteed
upper-bound error estimate for the first eigenfunction approximation and then a
guaranteed lower bound of the first eigenvalue.  The approach is based on complementary energy method from
\cite{HaslingerHlavacek,NeittaanmakiRepin,Repin,Vejchodsky_1,Vejchodsky_2} coupled with the upper and lower
 bounds of the eigenvalues by the conforming and nonconforming finite element methods.
The first eigenvalue is the key information in many practical applications such as Friedrichs, Poincar\'{e}, trace and
similar inequalities (cf. \cite{SebestovaVejchodsky}). Thus the two-sided bounds of the first eigenvalue of the partial
differential operators are very important.
Further, the proposed computable error estimates are asymptotically exact for the general eigenpair approximations
which are obtained by the conforming finite element method.
Based on this property, we can provide asymptotically lower bounds for general eigenvalues by the finite element method.
The most important feature and contribution of this paper are that the method can also provide the
reasonable accuracy even on the general regular meshes which is different from the existed methods.


An outline of the paper goes as follows. In Section \ref{Section_FEM}, we introduce the
finite element method for the eigenvalue problem and the corresponding basic
error estimates.  The computable error estimates for the eigenfunction approximations
and the corresponding upper-bound properties are given in Section \ref{Section_Upper_Bound}.
In Section \ref{Section_Lower_Bound}, lower bounds of eigenvalues are obtained
based on the results in Section \ref{Section_Upper_Bound}. Some numerical examples are presented
 to validate our theoretical analysis in Section \ref{Section_Numerical_Examples}.
Some concluding remarks are given in the last section.

\section{Finite element method for eigenvalue problem}\label{Section_FEM}
This section is devoted to introducing some notation and the finite element
method for eigenvalue problem. In this paper, the standard notation
for Sobolev spaces $H^s(\Omega)$ and $H({\rm div};\Omega)$ and their
associated norms and semi-norms \cite{Adams} will be used. We denote
$H_0^1(\Omega)=\{v\in H^1(\Omega):\ v|_{\partial\Omega}=0\}$,
where $v|_{\partial\Omega}=0$ is in the sense of trace. The letter $C$ (with or without subscripts)
denotes a generic positive constant which may be different at its different occurrences
 in  the paper.

For simplicity, this paper is concerned with the following model problem:
Find $(\lambda, u)$ such that
\begin{equation}\label{LaplaceEigenProblem}
\left\{
\begin{array}{rcl}
-\Delta u+ u&=&\lambda u, \quad {\rm in} \  \Omega,\\
u&=&0, \ \  \quad {\rm on}\  \partial\Omega,
\end{array}
\right.
\end{equation}
where $\Omega\subset\mathcal{R}^d$ $(d=2,3)$ is a bounded domain with
Lipschitz boundary $\partial\Omega$ and $\Delta$ denotes the Laplacian operator.
We will find that the method in this paper can easily be extended to more general
eigenvalue problems.

In order to use the finite element method to solve
the eigenvalue problem (\ref{LaplaceEigenProblem}), we need to define
the corresponding variational form as follows:
Find $(\lambda, u )\in \mathcal{R}\times V$ such that 
\begin{eqnarray}\label{weak_eigenvalue_problem}
a(u,v)&=&\lambda b(u,v),\quad \forall v\in V,
\end{eqnarray}
where $V:=H_0^1(\Omega)$ and
\begin{equation}\label{inner_product_a_b}
a(u,v)=\int_{\Omega}\big(\nabla u\cdot\nabla v + uv\big)d\Omega,
 \ \ \ \  \ \ b(u,v) = \int_{\Omega}uv d\Omega.
\end{equation}
The norms $\|\cdot\|_a$ and $\|\cdot\|_b$ are defined by
\begin{eqnarray*}
\|v\|_a=\sqrt{a(v,v)}\ \ \ \ \ {\rm and}\ \ \ \ \ \|v\|_b=\sqrt{b(v,v)}.
\end{eqnarray*}
It is well known that the eigenvalue problem (\ref{weak_eigenvalue_problem})
  has an eigenvalue sequence $\{\lambda_j \}$ (cf. \cite{BabuskaOsborn_Book,Chatelin}):
$$0<\lambda_1 < \lambda_2\leq\cdots\leq\lambda_k\leq\cdots,\ \ \
\lim_{k\rightarrow\infty}\lambda_k=\infty,$$ and associated
eigenfunctions
$$u_1, u_2, \cdots, u_k, \cdots,$$
where $b(u_i,u_j)=0$ when $i\neq j$. The first eigenvalue $\lambda_1$ is simple and
in the sequence $\{\lambda_j\}$, the $\lambda_j$ are repeated according to their
geometric multiplicity.

Now, we introduce the finite element method for the eigenvalue problem
(\ref{weak_eigenvalue_problem}). First we decompose the computing domain
$\Omega\subset \mathcal{R}^d\ (d=2,3)$
into shape-regular triangles or rectangles for $d=2$ (tetrahedrons or
hexahedrons for $d=3$) to produce the mesh $\mathcal{T}_h$ (cf. \cite{BrennerScott,Ciarlet}).
In this paper, we use $\mathcal{E}_h$ to denote the set of interior faces (edges or sides)
of $\mathcal{T}_h$.
The diameter of a cell $K\in\mathcal{T}_h$ is denoted by $h_K$ and
the mesh size $h$ describes  the maximum diameter of all cells
$K\in\mathcal{T}_h$. Based on the mesh $\mathcal{T}_h$, we can
construct a finite element space denoted by $V_h \subset V$.
For simplicity, we only consider the Lagrange type conforming finite element space
which is defined as follows
\begin{equation}\label{linear_fe_space}
V_h = \big\{ v_h \in C(\Omega)\ \big|\ v_h|_{K} \in \mathcal{P}_k,
\ \ \forall K \in \mathcal{T}_h\big\}\cap H_0^1(\Omega),
\end{equation}
where $\mathcal{P}_k$ denotes the space of polynomials of degree at most $k$.

We define the standard finite element scheme for the eigenvalue
 problem (\ref{weak_eigenvalue_problem}) as follows:
Find $(\lambda_h, u_h)\in \mathcal{R}\times V_h$
such that $b(u_h,u_h)=1$ and
\begin{eqnarray}\label{Weak_Eigenvalue_Discrete}
a(u_h,v_h)
&=&\lambda_h b(u_h,v_h),\quad\ \  \ \forall v_h\in V_h.
\end{eqnarray}
From \cite{BabuskaOsborn_1989,BabuskaOsborn_Book,Chatelin},
the discrete eigenvalue problem (\ref{Weak_Eigenvalue_Discrete}) has eigenvalues:
$$0<\lambda_{1,h}<\lambda_{2,h}\leq \cdots\leq \lambda_{k,h}
\leq\cdots\leq\lambda_{N_h,h},$$
and corresponding eigenfunctions
$$u_{1,h},\cdots, u_{k,h}, \cdots, u_{N_h,h},$$
where $b(u_{i,h},u_{j,h})=\delta_{ij}$ ($\delta_{ij}$ denotes the Kronecker function),
 when $1\leq i, j\leq N_h$ ($N_h$ is
the dimension of the finite element space $V_h$).

Let $M(\lambda_i)$ denote the eigenspace corresponding to the
eigenvalue $\lambda_i$ which is defined by
\begin{eqnarray*}
M(\lambda_i)&=&\big\{w\in H_0^1(\Omega): w\ {\rm is\ an\ eigenfunction\ of\
(\ref{weak_eigenvalue_problem})}\nonumber\\
&&\ \ \ \ \ \ \ \ \ \ \ \ \ \ \ \  {\rm corresponding\ to}\ \lambda_i \big\},
\end{eqnarray*}
and define
\begin{eqnarray}
\delta_h(\lambda_i)=\sup_{w\in M(\lambda_i), \|w\|_b=1}\inf_{v_h\in
V_h}\|w-v_h\|_{a}.
\end{eqnarray}

We also define the following quantity:
\begin{eqnarray}
\eta_{a}(h)&=&\sup_{f\in L^2(\Omega),\|f\|_b=1}\inf_{v_h\in V_h}\|Tf-v_h\|_{a},\label{eta_a_h_Def}
\end{eqnarray}
where $T:L^2(\Omega)\rightarrow V$ is defined as
\begin{equation}\label{laplace_source_operator}
a(Tf,v) = b(f,v), \ \ \ \ \  \forall f \in L^2(\Omega) \ \ \  {\rm and}\  \ \ \forall v\in V.
\end{equation}

Then the error estimates for the eigenpair approximations by the finite
element method can be described as follows.
\begin{lemma}(\cite[Lemma 3.6, Theorem 4.4]{BabuskaOsborn_1989} and \cite{Chatelin})
\label{Err_Eigen_Global_Lem}
There exists the exact eigenpair $(\lambda_i,u_i)$ of (\ref{weak_eigenvalue_problem}) such that
each eigenpair approximation
$(\lambda_{i,h},u_{i,h})$ $(i = 1, 2, \cdots, N_h)$ of
(\ref{Weak_Eigenvalue_Discrete}) has the following error estimates
\begin{eqnarray}
\|u_i-u_{i,h}\|_{a}
&\leq& \big(1+C_i\eta_a(h)\big)\delta_h(\lambda_i),\label{Err_Eigenfunction_Global_1_Norm} \\
\|u_i-u_{i,h}\|_{b}
&\leq& C_i\eta_{a}(h)\|u_i - u_{i,h}\|_{a},\label{Err_Eigenfunction_Global_0_Norm}\\
|\lambda_i-\lambda_{i,h}|
&\leq&C_i\|u_i - u_{i,h}\|_{a}^2\leq C_i\eta_a(h)\|u_i-u_{i,h}\|_a.\label{Estimate_Eigenvalue}
\end{eqnarray}
Here and hereafter $C_i$ is some constant depending on $i$ but independent of  the mesh size $h$.
\end{lemma}

\section{Complementarity based error estimate}\label{Section_Upper_Bound}
In this section, we derive a computable error estimate for the eigenfunction approximations
based on complementarity approach. A guaranteed upper bound of the error estimate for the first
 eigenfunction approximation is designed based on the lower bounds of the second eigenvalue.
 We also produce an asymptotically upper bound error estimate for the general eigenfunction approximations
  which are obtained by solving the discrete eigenvalue problem (\ref{Weak_Eigenvalue_Discrete}).

%


First, we recall the following divergence theorem
\begin{eqnarray}\label{Divergence_Equality}
\int_{\Omega}v{\rm div}\mathbf zd\Omega+\int_{\Omega}\mathbf z\cdot\nabla vd\Omega
=\int_{\partial\Omega}v\mathbf z\cdot\nu ds,\ \ \
\forall v\in V,\ \forall \mathbf z\in \mathbf W,
\end{eqnarray}
where $\mathbf W:=H({\rm div};\Omega)$ and $\nu$ denotes the unit outward normal to $\partial\Omega$.

We first give a guaranteed upper bound of the error estimate for the first eigenfunction
approximation and the method used here is independent from the way to obtain the solution.
We only consider the eigenfunction approximation $\widehat u_1\in V$ and estimate the error $e=u_1-\widehat u_1$
no matter how to obtain $\widehat u_1$. In this paper, we let $b(\widehat u_1,\widehat u_1)=1$
and the eigenvalue approximation $\widehat\lambda_1$
is determined as follows
\begin{eqnarray*}
\widehat\lambda_1=\frac{a(\widehat u_1,\widehat u_1)}{b(\widehat u_1,\widehat u_1)}=a(\widehat u_1,\widehat u_1).
\end{eqnarray*}
\begin{theorem}\label{Theorem_Upper_Bound}
Assume we have an eigenpair approximation $(\widehat\lambda_1,\widehat u_1)\in \mathcal{R}\times V$
corresponding to the first eigenvalue $\lambda_1$ and a lower bound eigenvalue approximation $\lambda_{2}^L$ of
the second eigenvalue $\lambda_2$ such that $\lambda_1\leq \widehat\lambda_1 < \lambda_2^L\leq \lambda_2$.
There exists an exact eigenfunction $u_1\in M(\lambda_1)$ such that the error estimate for
the first eigenfunction approximation $\widehat u_1\in V$ with $b(\widehat u_1,\widehat u_1)=1$ has the following guaranteed upper bound
\begin{eqnarray}\label{Upper_Bound}
\|u_1-\widehat u_1\|_a&\leq&  \frac{\lambda_{2}^L}{\lambda_{2}^L-\widehat \lambda_1}\eta(\widehat \lambda_1,\widehat u_1,\mathbf y),
\ \ \ \forall \mathbf y\in \mathbf W,
\end{eqnarray}
where $\eta(\widehat\lambda_1,\widehat u_1,\mathbf y)$ is defined as follows
\begin{eqnarray}\label{Definition_Eta}
\eta(\widehat\lambda_1,\widehat u_1,\mathbf y):=\big(\|\widehat\lambda_1\widehat u_1-\widehat u_1+{\rm div}\mathbf y\|_0^2+\|\mathbf y-\nabla \widehat u_1\|_0^2\big)^{1/2}.
\end{eqnarray}
\end{theorem}
\begin{proof}
We can choose $u_1\in M(\lambda_1)$ such that $b(v,u_1-\widehat u_1)=0$ for any $v\in M(\lambda_1)$.
Now we set $w=u_1-\widehat u_1$ and the following estimates hold
\begin{eqnarray}\label{Inequality_1}
&&a(u_1-\widehat u_1,w)-\widehat\lambda_1b(u_1-\widehat u_1,w)\nonumber\\
&=& \int_{\Omega}\lambda_1 u_1w d\Omega-\int_{\Omega}\nabla \widehat u_1\cdot\nabla wd\Omega
-\int_{\Omega}\widehat u_1wd\Omega-\widehat\lambda_1\int_{\Omega}u_1wd\Omega\nonumber\\
&&\ \ +\widehat\lambda_1\int_{\Omega}\widehat u_1wd\Omega+\int_{\Omega}w{\rm div}\mathbf yd\Omega
+\int_{\Omega}\mathbf y\cdot\nabla wd\Omega\nonumber\\
&=&\int_{\Omega}\big(\widehat\lambda_1\widehat u_1-\widehat u_1+{\rm div}\mathbf y\big)wd\Omega
+\int_{\Omega}\big(\mathbf y-\nabla \widehat u_1\big)\cdot\nabla wd\Omega\nonumber\\
&\leq& \|\widehat\lambda_1\widehat u_1-\widehat u_1+{\rm div}\mathbf y\|_0\|w\|_0+\|\mathbf y-\nabla \widehat u_1\|_0\|\nabla w\|_0\nonumber\\
&\leq& \big(\|\widehat\lambda_1\widehat u_1-\widehat u_1+{\rm div}\mathbf y\|_0^2+\|\mathbf y-\nabla \widehat u_1\|_0^2\big)^{1/2}\|w\|_a,
\ \ \ \forall\mathbf y\in \mathbf W.
\end{eqnarray}
Since $b(v,u_1-\widehat u_1)=0$ for any $v\in M(\lambda_1)$, the following inequalities hold
\begin{eqnarray}\label{Inequality_2}
\frac{\|w\|_a^2}{\|w\|_b^2}\geq \lambda_2\geq \lambda_{2}^L.
\end{eqnarray}
Combining (\ref{Inequality_1}) and (\ref{Inequality_2}) leads to the following estimate
\begin{eqnarray*}
\Big(1-\frac{\widehat\lambda_1}{\lambda_{2}^L}\Big)\|w\|_a^2 &\leq& \eta(\widehat\lambda_1,\widehat u_1,\mathbf y)\|w\|_a,
\ \ \ \forall\mathbf y\in \mathbf W.
\end{eqnarray*}
It means that we have
\begin{eqnarray*}
\|w\|_a&\leq& \frac{\lambda_{2}^L}{\lambda_{2}^L-\widehat\lambda_1}\eta(\widehat\lambda_1,\widehat u_1,\mathbf y),
\ \ \ \forall\mathbf y\in\mathbf W.
\end{eqnarray*}
This is the desired result (\ref{Upper_Bound}) and the proof is complete.
\end{proof}

A natural problem is to seek the minimization $\eta(\widehat\lambda,\widehat u,\mathbf y)$ over $\mathbf W$ for the fixed
eigenpair approximation $(\widehat\lambda,\widehat u)$. For this aim, we define the minimization problem:
Find $\mathbf y^*\in\mathbf  W$ such that
\begin{eqnarray}
\eta(\widehat\lambda, \widehat u,\mathbf y^*) \leq \eta(\widehat \lambda, \widehat u,\mathbf y),
\ \ \ \ \forall \mathbf y\in \mathbf W.
\end{eqnarray}
From \cite{Vejchodsky_1,Vejchodsky_2}, the optimization problem is equivalent to the following partial differential equation:
Find $\mathbf y^*\in \mathbf W$ such that
\begin{eqnarray}\label{Dual_Problem}
a^*(\mathbf y^*,\mathbf z)&=& \mathcal F^*(\widehat\lambda,\widehat u, \mathbf z),
\ \ \ \ \forall \mathbf z\in \mathbf W,
\end{eqnarray}
where
\begin{eqnarray*}
a^*(\mathbf y^*,\mathbf z)=\int_{\Omega}\big({\rm div}\mathbf y^*{\rm div}\mathbf z
+\mathbf y^*\cdot\mathbf z\big)d\Omega, \ \  \
\mathcal F^*(\widehat \lambda,\widehat u, \mathbf z)=-\int_{\Omega}\widehat\lambda\widehat u{\rm div}\mathbf zd\Omega.
\end{eqnarray*}
It is obvious $a^*(\cdot,\cdot)$ is an inner product in the space $\mathbf W$ and the corresponding norm
is $\||\mathbf z\||_{*}=\sqrt{a^*(\mathbf z,\mathbf z)}$. From the Riesz theorem,
we can know the dual problem (\ref{Dual_Problem})
has a unique solution.

Now, we state some properties for the estimator $\eta(\widehat \lambda,\widehat u,\mathbf y)$.
\begin{lemma}\label{Optimization_Property_Lemma}
Assume $\mathbf y^*$ be the solution of the dual problem (\ref{Dual_Problem}) and
 let $\widehat \lambda\in \mathcal{R}$,
$\widehat u\in V$ and $\mathbf y\in\mathbf W$ be arbitrary. Then the following equality holds
\begin{eqnarray}\label{Optimization_Property}
\eta^2(\widehat \lambda,\widehat u,\mathbf y)&=&\eta^2(\widehat \lambda,\widehat u,\mathbf y^*)
+\||\mathbf y^*-\mathbf y\||_*^2.
\end{eqnarray}
\end{lemma}

In order to give a computable error estimate, the reasonable choice is a certain approximate
solution $\mathbf y_h\in \mathbf W$ of the dual problem (\ref{Dual_Problem}). Then we can give a
guaranteed upper bound of the error estimate for the first eigenfunction approximation.
\begin{corollary}
Under the conditions of Theorem \ref{Theorem_Upper_Bound},
 there exists an exact eigenfunction $u_1\in M(\lambda_1)$
such that the error estimate for the eigenpair approximation
$(\widehat \lambda_1,\widehat u_1)$ has the following upper bound
\begin{eqnarray}\label{Upper_Bound_Computable}
\|u_1-\widehat u_1\|_a&\leq& \frac{\lambda_{2}^L}{\lambda_{2}^L-\widehat \lambda_1}
\eta(\widehat \lambda_1,\widehat u_1,\mathbf y_h),
\end{eqnarray}
where $\mathbf y_h\in \mathbf W$ is a reasonable approximate solution
of the dual problem (\ref{Dual_Problem})
with  $\widehat \lambda=\widehat \lambda_1$ and $\widehat u=\widehat u_1$.
\end{corollary}

We would like to point out that the quantity $\eta(\lambda_{i,h},u_{i,h},\mathbf y^*)$, where $\mathbf y^*\in \mathbf W$
is the solution of (\ref{Dual_Problem}) with $\widehat\lambda =\lambda_{i,h}$ and $\widehat u=u_{i,h}$, is
an asymptotically  exact error estimate for the
eigenfunction approximation $u_{i,h}$ when the eigenpair approximation is obtained by solving
the discrete eigenvalue problem (\ref{Weak_Eigenvalue_Discrete}). Now, let us
discuss the efficiency of the a posteriori error estimate $\eta(\lambda_{i,h},u_{i,h},\mathbf y^*)$ and
$\eta(\lambda_{i,h},u_{i,h},\mathbf y_h)$.
\begin{theorem}\label{Efficiency_Theorem}
Assume $(\lambda_{i,h},u_{i,h})$ be an eigenpair approximation of the discrete eigenvalue
problem (\ref{Weak_Eigenvalue_Discrete})
corresponding to the eigenvalue $\lambda_i$. Then there exists an exact eigenfunction $u_i\in M(\lambda_i)$
such that $\eta(\lambda_{i,h},u_{i,h},\mathbf y^*)$ has the following inequalities
\begin{eqnarray}\label{Efficiency}
\theta_{1,i}\|u_i-u_{i,h}\|_a \leq
\eta(\lambda_{i,h},u_{i,h},\mathbf y^*)\leq \theta_{2,i}\|u_i-u_{i,h}\|_a,
\end{eqnarray}
where $\mathbf y^*\in \mathbf W$ is the solution  of the dual problem (\ref{Dual_Problem})
with  $\widehat \lambda= \lambda_{i,h}$ and $\widehat u=u_{i,h}$ and
\begin{eqnarray}\label{Theta_1_and_2}
\theta_{1,i}:=(1-C_i^2\lambda_{i,h}\eta_a^2(h))\ \ \ {\rm and}\ \ \
\theta_{2,i}:=\sqrt{1+\big(2(\lambda_i-1)^2+1\big)C_i^2\eta_a^2(h)}.
\end{eqnarray}
Further, we have the following asymptotic exactness
\begin{eqnarray}\label{Exactness}
\lim_{h\rightarrow 0}\frac{\eta(\lambda_{i,h},u_{i,h},\mathbf y^*)}{\|u_i-u_{i,h}\|_a}=1.
\end{eqnarray}
\end{theorem}
\begin{proof}
Similarly, we can also choose $u_i\in M(\lambda_i)$ such that $b(v,u_i-u_{i,h})=0$ for any $v\in M(\lambda_i)$.
Then from the similar process in  (\ref{Inequality_1}), we have
\begin{eqnarray}
&&\|u_i-u_{i,h}\|_a^2\leq \eta(\lambda_{i,h},u_{i,h},\mathbf y)\|u_i-u_{i,h}\|_a
+\lambda_{i,h}\|u_i-u_{i,h}\|_b^2\nonumber\\
&&\leq \eta(\lambda_{i,h},u_{i,h},\mathbf y)\|u_i-u_{i,h}\|_a
+C_i^2\lambda_{i,h}\eta_a^2(h)\|u_i-u_{i,h}\|_a^2,\ \ \ \forall \mathbf y\in\mathbf W.
\end{eqnarray}
It leads to
\begin{eqnarray}\label{Inequality_6}
\|u_i-u_{i,h}\|_a &\leq& \frac{1}{1-C_i^2\lambda_{i,h}\eta_a^2(h)}\eta(\lambda_{i,h},u_{i,h},\mathbf y),
\ \ \ \forall \mathbf y\in\mathbf W.
\end{eqnarray}

From the definition (\ref{Definition_Eta}), the eigenvalue problem (\ref{LaplaceEigenProblem})
and $\nabla u_i\in\mathbf W$, we have
\begin{eqnarray}\label{Inequality_4}
\eta^2(\lambda_{i,h},u_{i,h},\nabla u_i)=\|\nabla u_{i,h}-\nabla u_i\|_b^2
+\|(\lambda_{i,h}-1)u_{i,h}-(\lambda_i-1)u_i\|_b^2.
\end{eqnarray}
Then combining (\ref{Optimization_Property}), (\ref{Inequality_4}) and Lemma \ref{Err_Eigen_Global_Lem},
 the following estimates hold
\begin{eqnarray}\label{Inequality_3}
&&\eta^2(\lambda_{i,h},u_{i,h},\mathbf y^*)\leq \eta^2(\lambda_{i,h},u_{i,h},\nabla u_i)\nonumber\\
&=&\|\nabla u_{i,h}-\nabla u_i\|_b^2+\|(\lambda_{i,h}-1)u_{i,h}-(\lambda_i-1)u_i\|_b^2\nonumber\\
&=&\|u_i-u_{i,h}\|_a^2+\|(\lambda_{i,h}-1)u_{i,h}-(\lambda_i-1)u_i\|_b^2-\|u_i-u_{i,h}\|_b^2\nonumber\\
&=& \|u_i-u_{i,h}\|_a^2+\|(\lambda_{i,h}-\lambda_i)u_{i,h}+(\lambda_i-1)(u_{i,h}-u_i)\|_b^2
-\|u_i-u_{i,h}\|_b^2\nonumber\\
&\leq& \|u_i-u_{i,h}\|_a^2 + 2|\lambda_{i,h}-\lambda_i|^2 +\big(2(\lambda_i-1)^2-1\big)\|u_i-u_{i,h}\|_b^2\nonumber\\
&\leq&\big(1+2C_i^2\eta_a^2(h)+\big(2(\lambda_i-1)^2-1\big)C_i^2\eta_a^2(h)\big)\|u_i-u_{i,h}\|_a^2\nonumber\\
&\leq&\big(1+\big(2(\lambda_i-1)^2+1\big)C_i^2\eta_a^2(h)\big)\|u_i-u_{i,h}\|_a^2,
\end{eqnarray}
where we used the estimate
\begin{eqnarray*}
\lambda_{i,h}-\lambda_i &\leq& C_i\eta_a(h)\|u_i-u_{i,h}\|_a.
\end{eqnarray*}
The inequality (\ref{Inequality_3}) leads to the following estimate
\begin{eqnarray}\label{Inequality_5}
\eta(\lambda_{i,h},u_{i,h},\mathbf y^*)&\leq&\sqrt{1+\big(2(\lambda_i-1)^2+1\big)C_i^2\eta_a^2(h)}\|u_i-u_{i,h}\|_a.
\end{eqnarray}
From inequalities (\ref{Optimization_Property}), (\ref{Inequality_6}) and (\ref{Inequality_5}), we
obtain the desired result  (\ref{Efficiency}) and (\ref{Exactness}) can be deduced easily from the
fact that $\eta_a(h)\rightarrow 0$ as $h\rightarrow 0$.
\end{proof}

\begin{corollary}\label{Efficiency_h_Eigenfun_Corollary}
Assume the conditions of Theorem \ref{Efficiency_Theorem} hold and there exists a constant $\gamma_i>0$ such that
the approximation $\mathbf y_h$ of $\mathbf y^*$ satisfies
$\||\mathbf y^*-\mathbf y_h\||_*\leq \gamma_i\|u_i-u_{i,h}\|_a$. Then the following efficiency holds
\begin{eqnarray}\label{Efficiency_2}
\eta(\lambda_{i,h},u_{i,h},\mathbf y_h)&\leq&\sqrt{\theta_{2,i}^2+\gamma_i^2}\|u_i-u_{i,h}\|_a.
\end{eqnarray}
Further, the estimator $\eta(\lambda_{i,h},u_{i,h},\mathbf y_h)$ is asymptotically exact if and only if
the following condition holds
\begin{eqnarray}\label{Condition_Exact}
\lim_{h\rightarrow0}\frac{\||\mathbf y^*-\mathbf y_h\||_*}{\|u_i-u_{i,h}\|_a}=0.
\end{eqnarray}
\end{corollary}
\begin{proof}
First from (\ref{Optimization_Property}) and (\ref{Efficiency}), we have
\begin{eqnarray}
\eta^2(\lambda_{i,h},u_{i,h},\mathbf y_h)&=&\eta^2(\lambda_{i,h},u_{i,h},\mathbf y^*)
+\||\mathbf y^*-\mathbf y_h\||_*^2\nonumber\\
&\leq&\theta_{2,i}^2\|u_i-u_{i,h}\|_a^2+\gamma_i^2\|u_i-u_{i,h}\|_a^2\nonumber\\
&\leq&(\theta_{2,i}^2+\gamma_i^2)\|u_i-u_{i,h}\|_a^2.
\end{eqnarray}
Then the desired result (\ref{Efficiency_2}) can be obtained and the asymptotically exactness of the estimator
follows immediately from the condition (\ref{Condition_Exact}).
\end{proof}

\section{Lower bound of the eigenvalue}\label{Section_Lower_Bound}
In this section, based on the guaranteed upper bound for the error estimate of the first eigenfunction approximation,
we give a guaranteed lower bound of the first eigenvalue.
Further, we also give asymptotically lower bounds of the general eigenvalues based on the asymptotically exact
 error estimates for the general eigenfunction approximations which are obtained by solving the
 discrete finite element eigenvalue problem (\ref{Weak_Eigenvalue_Discrete}).
 Actually, the process is very direct since we have
 the following Rayleigh quotient expansion which comes from \cite{BabuskaOsborn_1989,BabuskaOsborn_Book}.
\begin{lemma}(\cite{BabuskaOsborn_1989,BabuskaOsborn_Book})\label{Rayleigh_Quotient_error_theorem}
Assume $(\lambda,u)$ is an exact solution of the eigenvalue problem
(\ref{weak_eigenvalue_problem}) and  $0\neq \psi\in V$. Let us define
\begin{eqnarray}\label{rayleighw}
\bar{\lambda}=\frac{a(\psi,\psi)}{b(\psi,\psi)}.
\end{eqnarray}
Then we have
\begin{eqnarray}\label{rayexpan}
\bar{\lambda}-\lambda
&=&\frac{a(u-\psi,u-\psi)}{b(\psi,\psi)}-\lambda
\frac{b(u-\psi,u-\psi)}{b(\psi,\psi)}.
\end{eqnarray}
\end{lemma}

\begin{theorem}\label{Guaranteed_Lower_Bound_Theorem}
Assume $\lambda_1$ is the first eigenvalue of the eigenvalue problem (\ref{LaplaceEigenProblem})
and $(\widehat \lambda_1,\widehat u_1)\in \mathcal{R}\times V$ ($\|\widehat u_1\|_b=1$)
be the eigenpair approximation for the first eigenvalue and eigenfunction, respectively.
Then we have the following
guaranteed lower bound of the first eigenvalue
\begin{eqnarray}\label{Upper_Bound_Estimate_Lambda}
\widehat\lambda_1-\lambda_1 &\leq& \left(\frac{\lambda_{2}^L}{\lambda_{2}^L-\widehat\lambda_1}\right)\frac{\lambda_2^L}{\lambda_2^L-\alpha^2\eta^2(\widehat\lambda_1,\widehat u_1,\mathbf y_h)}
\eta^2(\widehat\lambda_1,\widehat u_1,\mathbf y_h),
\end{eqnarray}
where $\alpha=\lambda_{2}^L/(\lambda_{2}^L-\widehat\lambda_1)$ and $\mathbf y_h\in \mathbf W$ is a reasonable approximate solution  of the dual problem (\ref{Dual_Problem})
with  $\widehat \lambda=\widehat \lambda_1$ and $\widehat u=\widehat u_1$.

Then the following guaranteed lower-bound result holds
\begin{eqnarray}\label{Lower_Bound_Lambda}
\widehat\lambda_1^L:= \widehat\lambda_1-\left(\frac{\lambda_{2}^L}{\lambda_{2}^L-\widehat\lambda_1}\right)\frac{\lambda_2^L}{\lambda_2^L-\alpha^2\eta^2(\widehat\lambda_1,\widehat u_1,\mathbf y_h)}
\eta^2(\widehat\lambda_1,\widehat u_1,\mathbf y_h)
\leq  \lambda_1,
\end{eqnarray}
where $\widehat\lambda_1^L$ denotes a lower bound of the first eigenvalue $\lambda_1$.
\end{theorem}
\begin{proof}
Similarly, we can also choose $u_1\in M(\lambda_1)$ such that $b(v,u_1-\widehat u_1)=0$ for any $v\in M(\lambda_1)$.
We also set $w=u_1-\widehat u_1$ and from Lemma \ref{Rayleigh_Quotient_error_theorem}, (\ref{Upper_Bound}), (\ref{Inequality_1})
and $\|\widehat u_1\|_b=1$, we have
\begin{eqnarray}\label{Inequality_9}
\widehat\lambda_1-\lambda_1 -(\widehat\lambda_1-\lambda_1)\|w\|_b^2 &=&
a(u_1-\widehat u_1,u_1-\widehat u_1)-\widehat\lambda_1 b(u_1-\widehat u_1,u_1-\widehat u_1)\nonumber\\
&\leq&\eta(\widehat\lambda_1,\widehat u_1,\mathbf y_h)\|u_1-\widehat u_1\|_a.
\end{eqnarray}
Combining (\ref{Upper_Bound}), (\ref{Inequality_2}) and (\ref{Inequality_9}) leads to the following inequalities
\begin{eqnarray}
\widehat\lambda_1-\lambda_1 &\leq& \frac{\|w\|_a}{1-\|w\|_b^2}\eta(\widehat\lambda_1,\widehat u_1,\mathbf y_h)\nonumber\\
&\leq& \frac{\|w\|_a}{1-\frac{1}{\lambda_2^L}\|w\|_a^2}\eta(\widehat\lambda_1,\widehat u_1,\mathbf y_h)\nonumber\\
&\leq& \alpha\frac{\lambda_2^L}{\lambda_2^L-\alpha^2\eta^2(\widehat\lambda_1,\widehat u_1,\mathbf y_h)}
\eta^2(\widehat\lambda_1,\widehat u_1,\mathbf y_h).
\end{eqnarray}
This is the desired result (\ref{Upper_Bound_Estimate_Lambda}). The lower bound result (\ref{Lower_Bound_Lambda}) holds directly and
the proof is complete.
\end{proof}
\begin{remark}
From above derivation (Theorems \ref{Theorem_Upper_Bound} and \ref{Guaranteed_Lower_Bound_Theorem}),
it is easy to know the current method here can also obtain the guaranteed lower bounds for the first $m$
eigenvalues if provided the separation condition $\lambda_m< \lambda_{m+1}^L\leq \lambda_{m+1}$
and $\lambda_{m+1}^L$ is known.
\end{remark}
\begin{theorem}\label{Efficiency_Eigenvalue_Theorem}
Assume the conditions of Theorem \ref{Efficiency_Theorem} hold.
Then the following inequalities hold
\begin{eqnarray}\label{Efficiency_Eigenvalue}
\frac{1-\lambda_i C_i^2\eta_a^2(h)}{\theta_{2,i}^2}\eta^2(\lambda_{i,h},u_{i,h},\mathbf y^*)
\leq \lambda_{i,h}-\lambda_i\leq  \frac{1}{\theta_{1,i}^2}\eta^2(\lambda_{i,h},u_{i,h},\mathbf y^*),
\end{eqnarray}
where $\mathbf y^*\in \mathbf W$ is the solution  of the dual problem (\ref{Dual_Problem})
with  $\widehat \lambda= \lambda_{i,h}$ and $\widehat u=u_{i,h}$.

Further, we have the following asymptotic exactness
\begin{eqnarray}\label{Exactness_Eigenvalue}
\lim_{h\rightarrow 0}\frac{\lambda_{i,h}-\lambda_i}{\eta^2(\lambda_{i,h},u_{i,h},\mathbf y^*)}=1.
\end{eqnarray}
\end{theorem}
\begin{proof}
From Lemma \ref{Err_Eigen_Global_Lem}, (\ref{Efficiency}) and (\ref{rayexpan}), we have
\begin{eqnarray}\label{Inequality_7}
\lambda_{i,h}-\lambda_i&=&\|u_i-u_{i,h}\|_a^2-\lambda_i\|u_i-u_{i,h}\|_b^2\nonumber\\
&\geq&\|u_i-u_{i,h}\|_a^2-\lambda_i C_i^2\eta_a^2(h)\|u_i-u_{i,h}\|_a^2\nonumber\\
&=&(1-\lambda_i C_i^2\eta_a^2(h))\|u_i-u_{i,h}\|_a^2\nonumber\\
&\geq&\frac{1-\lambda_i C_i^2\eta_a^2(h)}{\theta_{2,i}^2}\eta^2(\lambda_{i,h},u_{i,h},\mathbf y^*).
\end{eqnarray}
From (\ref{Efficiency}) and (\ref{rayexpan}), the following inequalities hold
\begin{eqnarray}\label{Inequality_8}
\lambda_{i,h}-\lambda_i &\leq& \|u_i-u_{i,h}\|_a^2\leq \frac{1}{\theta_{1,i}^2}\eta^2(\lambda_{i,h},u_{i,h},\mathbf y^*).
\end{eqnarray}
The desired result (\ref{Efficiency_Eigenvalue}) can be obtained by
combining (\ref{Inequality_7}) and (\ref{Inequality_8}).
Then we can deduce the asymptotic exactness easily by (\ref{Efficiency_Eigenvalue})
and the property $\eta_a(h)\rightarrow 0$ as $h\rightarrow 0$.
\end{proof}
Based on the result (\ref{Efficiency_Eigenvalue}), we can produce an asymptotically lower bound for the
general eigenvalue $\lambda_i$ by the finite element method.
\begin{corollary}\label{Lower_Bound_Eigen_Corollary}
Under the conditions of Theorem \ref{Efficiency_Eigenvalue_Theorem}, when the mesh size $h$ is small enough,
the following asymptotically lower bound for each eigenvalue $\lambda_i$ holds
\begin{eqnarray}\label{Lower_Bound_Eigen_i}
\lambda_{i,h}^{L}:=\lambda_{i,h}-\kappa \eta^2(\lambda_{i,h},u_{i,h},\mathbf y_h) \leq \lambda_i,
\end{eqnarray}
where $\kappa$ is a number larger than $1$ and
$\mathbf y_h\in \mathbf W$ is a reasonable approximate solution  of the dual problem (\ref{Dual_Problem})
with  $\widehat \lambda= \lambda_{i,h}$ and $\widehat u=u_{i,h}$.
\end{corollary}
\begin{proof}
From Lemma \ref{Optimization_Property_Lemma} and (\ref{Efficiency_Eigenvalue}), we have the following
inequalities
\begin{eqnarray*}
\lambda_{i,h}-\lambda_i\leq  \frac{1}{\theta_{1,i}^2}\eta^2(\lambda_{i,h},u_{i,h},\mathbf y^*) \leq
 \frac{1}{\theta_{1,i}^2}\eta^2(\lambda_{i,h},u_{i,h},\mathbf y_h).
\end{eqnarray*}
Combining (\ref{Theta_1_and_2}) and $\eta_a(h)\rightarrow 0$ as $h\rightarrow 0$ leads to
$\theta_{1,i}^2\rightarrow 1$ as  $h\rightarrow 0$. Then the lower bound result
(\ref{Lower_Bound_Eigen_i}) holds when the mesh size $h$ is small enough.
\end{proof}
\begin{remark}\label{Lower_Bound_Eigen_Remark}
It is easy to know that if we choose $\kappa$ closer to $1$, the mesh size $h$ need to be smaller.
For example, we can choose $\kappa =2$ and has the following eigenvalue approximation
\begin{eqnarray*}
\lambda_{i,h}^{L}:=\lambda_{i,h}-2 \eta^2(\lambda_{i,h},u_{i,h},\mathbf y_h),
\end{eqnarray*}
which is a lower bound of the eigenvalue $\lambda_i$ when $h$ is small enough.
\end{remark}
\begin{corollary}\label{Efficiency_h_Eigenvalue_Corollary}
Assume the conditions of Corollary \ref{Lower_Bound_Eigen_Corollary} hold and there exists a constant $\gamma_i$
such that the approximation $\mathbf y_h$ of $\mathbf y^*$ satisfies
$\||\mathbf y^*-\mathbf y_h\||_*\leq \gamma_i\|u_i-u_{i,h}\|_a$. Then the following efficiency holds
\begin{eqnarray}\label{Efficiency_2_Eigenvalue}
\eta^2(\lambda_{i,h},u_{i,h},\mathbf y_h)&\leq&\Big(1+\frac{\gamma_i^2}{\theta_{1,i}^2}\Big)
\frac{\theta_{2,i}^2}{1-\lambda_i C_i^2\eta_a^2(h)}(\lambda_{i,h}-\lambda_i).
\end{eqnarray}
Further, the estimator $\eta^2(\lambda_{i,h},u_{i,h},\mathbf y_h)$ is asymptotically exact for
the eigenvalue error $\lambda_{i,h}-\lambda_i$ if and only if the condition (\ref{Condition_Exact})
holds.
\end{corollary}
\begin{proof}
First from (\ref{Optimization_Property}), (\ref{Efficiency}) and (\ref{Efficiency_Eigenvalue}),
we have the following estimates
\begin{eqnarray*}
\eta^2(\lambda_{i,h},u_{i,h},\mathbf y_h)&=&\eta^2(\lambda_{i,h},u_{i,h},\mathbf y^*)
+\||\mathbf y^*-\mathbf y_h\||_*^2\nonumber\\
&\leq& \eta^2(\lambda_{i,h},u_{i,h},\mathbf y^*)+\gamma_i^2\|u_i-u_{i,h}\|_a^2\nonumber\\
&\leq& \eta^2(\lambda_{i,h},u_{i,h},\mathbf y^*)+\frac{\gamma_i^2}{\theta_{1,i}^2}
\eta^2(\lambda_{i,h},u_{i,h},\mathbf y^*)\nonumber\\
&\leq&\Big(1+\frac{\gamma_i^2}{\theta_{1,i}^2}\Big)\eta^2(\lambda_{i,h},u_{i,h},\mathbf y^*)\nonumber\\
&\leq&\Big(1+\frac{\gamma_i^2}{\theta_{1,i}^2}\Big)
\frac{\theta_{2,i}^2}{1-\lambda_i C_i^2\eta_a^2(h)}(\lambda_{i,h}-\lambda_i).
\end{eqnarray*}
This is the desired result (\ref{Efficiency_2_Eigenvalue}) and the asymptotically exactness
result follows immediately from the condition (\ref{Condition_Exact}).
\end{proof}
\begin{remark}
From Corollaries \ref{Efficiency_h_Eigenfun_Corollary} and \ref{Efficiency_h_Eigenvalue_Corollary},
the estimators
$\eta(\lambda_{i,h},u_{i,h},\mathbf y_h)$ and $\eta^2(\lambda_{i,h},u_{i,h},\mathbf y_h)$
are asymptotically
exact for $\|u-u_{i,h}\|_a$ and $\lambda_{i,h}-\lambda_i$, respectively,
when the condition $\lim\limits_{h\rightarrow 0}\gamma_i=0$ holds.
\end{remark}
\section{Numerical results}\label{Section_Numerical_Examples}
In this section, two numerical examples are presented to validate the efficiency of the
posteriori estimate, the upper bound of the error estimate and lower bound of the first eigenvalue
proposed in this paper.

In order to give the a posteriori error estimate $\eta(\lambda_{i,h},u_{i,h},\mathbf y_h)$, we need to
solve the dual problem (\ref{Dual_Problem}) to produce the approximation $\mathbf y_h$ of $\mathbf y^*$.
Here, the dual problem (\ref{Dual_Problem}) is solved
using the same mesh $\mathcal{T}_h$.
We solve the dual problem (\ref{Dual_Problem})
to obtain an approximation $\mathbf y_h^*\in \mathbf W_h\subset \mathbf W$ with the $H({\rm div};\Omega)$
conforming finite element space $\mathbf W_h$ defined as follows \cite{BrezziFortin}
\begin{eqnarray}
\mathbf W_h^p=\big\{\mathbf w\in \mathbf W:\ \mathbf w|_K\in {\rm RT}_p,\ \forall K\in \mathcal{T}_h\big\},
\end{eqnarray}
where ${\rm RT}_p= (\mathcal{P}_p)^d+\mathbf x\mathcal{P}_p$. Then the approximate solution $\mathbf y_h^p\in \mathbf W_h^p$
of the dual problem (\ref{Dual_Problem}) is defined as follows: Find $\mathbf y_h^* \in \mathbf W_h^p$ such that
\begin{eqnarray}\label{Dual_Problem_Discrete}
a^*(\mathbf y_h^*,\mathbf z_h)&=&\mathcal{F}^*(\lambda_{i,h},u_{i,h},\mathbf z_h),\ \ \ \forall \mathbf z_h\in \mathbf W_h^p.
\end{eqnarray}
After obtaining $\mathbf y_h^*$, we can compute the a posteriori error estimate $\eta(\lambda_{i,h},u_{i,h},\mathbf y_h^*)$
as in (\ref{Definition_Eta}).

We can obtain the lower bound $\lambda_{2,h}^L$ of the second eigenvalue $\lambda_2$ by the nonconforming
finite element method from the papers \cite{CarstensenGedicke,Liu,SebestovaVejchodsky}. Based on $\lambda_{2,h}^L$,
we can compute the guaranteed upper bound of the error estimate for the  first eigenfunction approximation $u_{1,h}$ as
\begin{eqnarray*}
\eta_h^U(\lambda_{1,h},u_{1,h},\mathbf y_h^*):=
\frac{\lambda_{2,h}^L}{\lambda_{2,h}^L-\lambda_{1,h}}\eta(\lambda_{1,h},u_{1,h},\mathbf y_h^*),
\end{eqnarray*}
and the guaranteed lower bound of the first eigenvalue $\lambda_1$ as follows
\begin{eqnarray*}
\lambda_{1,h}^L:= \lambda_{1,h}-\left(\frac{\lambda_{2}^L}{\lambda_{2}^L-\widehat\lambda_1}\right)\frac{\lambda_2^L}{\lambda_2^L-\alpha^2\eta^2(\widehat\lambda_1,\widehat u_1,\mathbf y_h)}
\eta^2(\widehat\lambda_1,\widehat u_1,\mathbf y_h)
\leq  \lambda_1,
\end{eqnarray*}
where $\alpha=\lambda_{2}^L/(\lambda_{2}^L-\lambda_{1,h})$. 

In this paper, we solve the eigenvalue problem by the multigrid method from the papers \cite{Xie_JCP,Xie_IMA}
which only needs the optimal memory and computational complexity.

\subsection{Eigenvalue problem on unit square}
In the first example, we solve the eigenvalue problem (\ref{weak_eigenvalue_problem})
on the unit square $\Omega=(0,1)\times (0,1)$. In order to investigate the efficiency of the
a posteriori error estimate $\eta(\lambda_{i,h},u_{i,h},\mathbf y_h^*)$, the guaranteed upper bound
$\eta_h^U(\lambda_{1,h},u_{1,h},\mathbf y_h^*)$
of the error estimate $\|u_1-u_{1,h}\|_a$ and the lower bound $\lambda_{1,h}^L$ of the first eigenvalue $\lambda_1$,
we produce the sequence of finite element spaces on the sequence of meshes
which are obtained by the regular refinement (connecting the midpoints of each edge) from an initial
mesh. In this example, the initial mesh is showed in Figure \ref{Exam_1_Initial_Mesh}
which is generated by Delaunay method.

First we solve the eigenvalue
problem (\ref{Weak_Eigenvalue_Discrete})
by the linear conforming finite element method and solve the dual problem (\ref{Dual_Problem_Discrete})
in the finite element space $\mathbf W_h^0$ and $\mathbf W_h^1$, respectively.
The corresponding numerical results are presented in Figure \ref{Exam_1_P_1_RT0_RT1} which
shows that the a posteriori error estimate $\eta(\lambda_{1,h},u_{1,h},\mathbf y_h^*)$
is efficient when we solve the dual problem by  $\mathbf W_h^1$. Figure \ref{Exam_1_P_1_RT0_RT1} also shows
the validation of the guaranteed upper bound $\eta_h^U(\lambda_{1,h},u_{1,h},\mathbf y_h^*)$ for the error $\|u_1-u_{1,h}\|_a$
and the eigenvalue approximation $\lambda_{1,h}^L$ is really a guaranteed lower bound for the first eigenvalue
 $\lambda_1=1+2\pi^2$ despite the way to solve the dual problem by $\mathbf W_h^0$ or $\mathbf W_h^1$.

\begin{figure}[htbp]
\centering
\includegraphics[width=5.5cm,height=5.5cm]{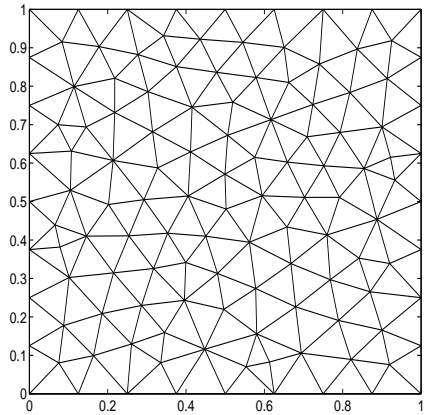}
\caption{\small\texttt The initial mesh for the unit square} \label{Exam_1_Initial_Mesh}
\end{figure}

\begin{figure}[ht]
\centering
\includegraphics[width=6.5cm,height=6cm]{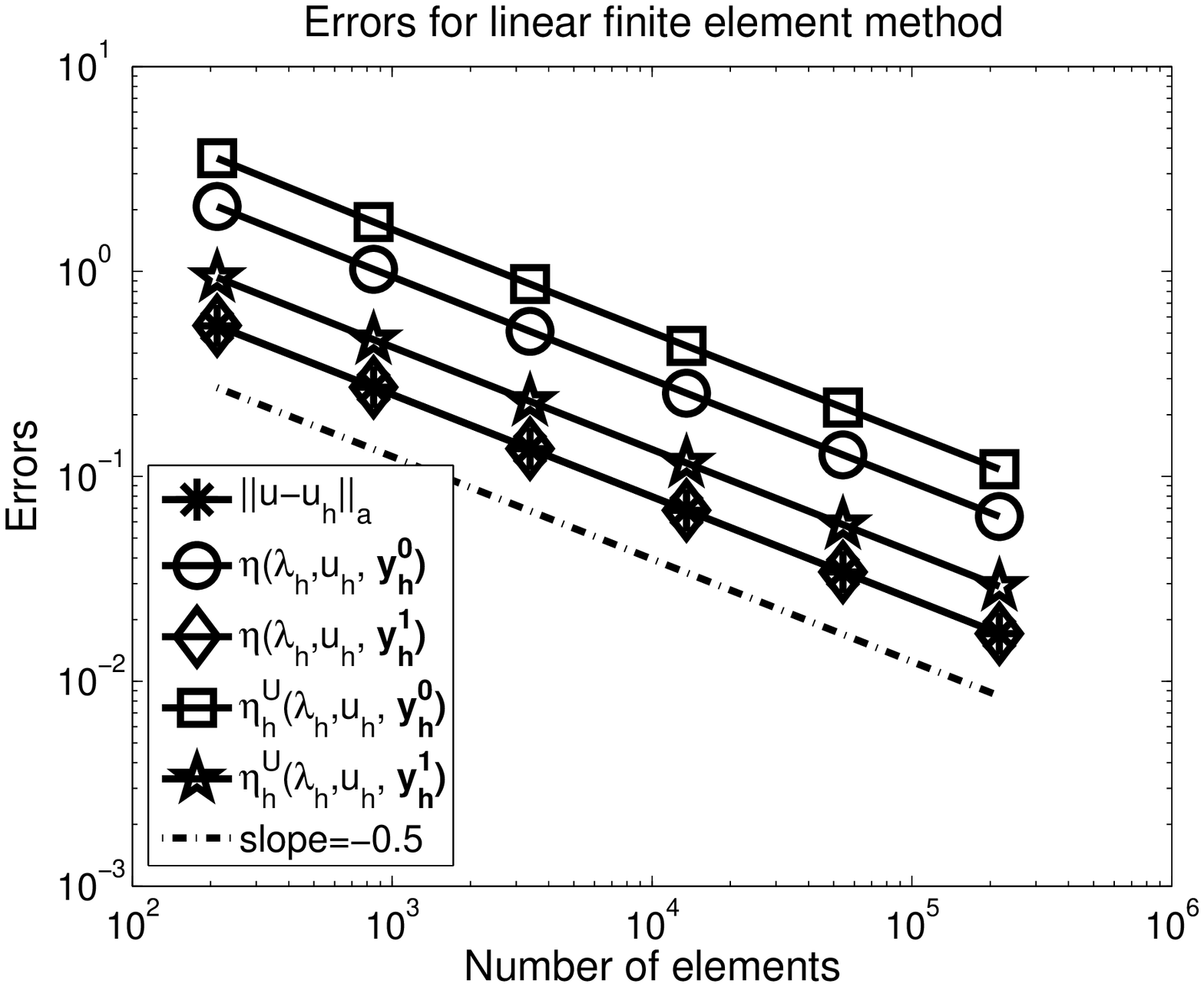}
\includegraphics[width=6.5cm,height=6cm]{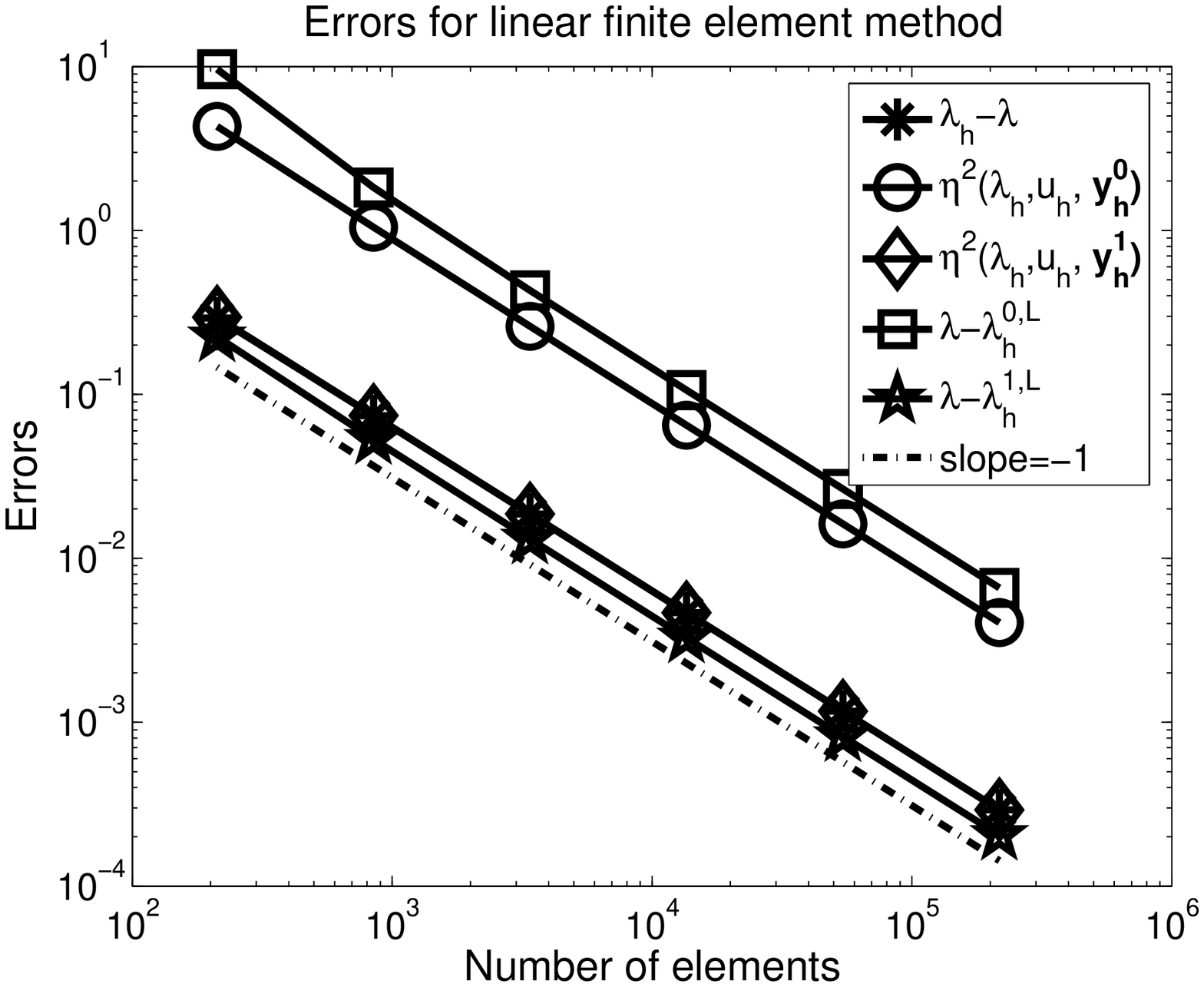}
\caption{\small\texttt The errors for the unit square domain when the eigenvalue problem is solved
by the linear finite element method, where $\eta(\lambda_h,u_h,\mathbf y_h^0)$ and
$\eta(\lambda_h,u_h,\mathbf y_h^1)$ denote the a posteriori error estimates $\eta(\lambda_{1,h},u_{1,h},\mathbf y_h^*)$
when the dual problem is solved by $\mathbf W_h^0$ and $\mathbf W_h^1$, respectively, and $\lambda_h^{0,L}$ and
$\lambda_h^{1,L}$ denote the guaranteed lower bounds of the first eigenvalue $\lambda_1$ when the dual problem is
solved by $\mathbf W_h^0$ and $\mathbf W_h^1$, respectively}
\label{Exam_1_P_1_RT0_RT1}
\end{figure}

We also solve the eigenvalue problem (\ref{Weak_Eigenvalue_Discrete}) by the quadratic
finite element method and solve the dual problem (\ref{Dual_Problem_Discrete})
with the finite element space $\mathbf W_h^1$ and $\mathbf W_h^2$, respectively.
Figure \ref{Exam_1_P_2_RT1_RT2} shows the corresponding numerical results. From Figure \ref{Exam_1_P_2_RT1_RT2},
we can find that the a posteriori error estimate $\eta(\lambda_{1,h},u_{1,h},\mathbf y_h^*)$ is efficient
when we solve the dual problem by  $\mathbf W_h^2$. Figure \ref{Exam_1_P_2_RT1_RT2} also shows
$\eta_h^U(\lambda_{1,h},u_{1,h},\mathbf y_h^*)$  is really the guaranteed upper bound of the error $\|u_1-u_{1,h}\|_a$
and the eigenvalue approximation $\lambda_{1,h}^L$ is also really a guaranteed lower bound of the first eigenvalue
 $\lambda_1$.

\begin{figure}[ht]
\centering
\includegraphics[width=6.5cm,height=6cm]{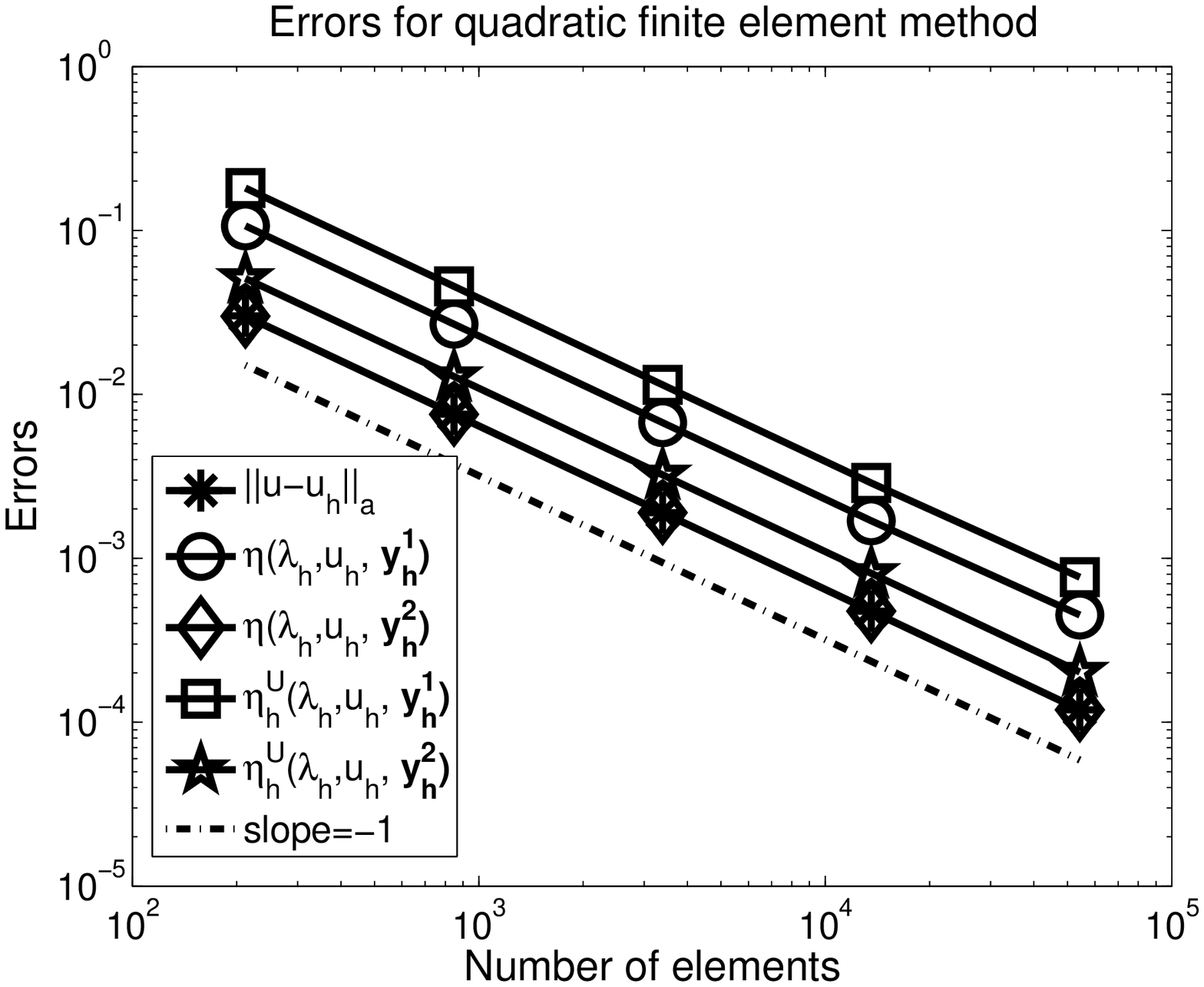}
\includegraphics[width=6.5cm,height=6cm]{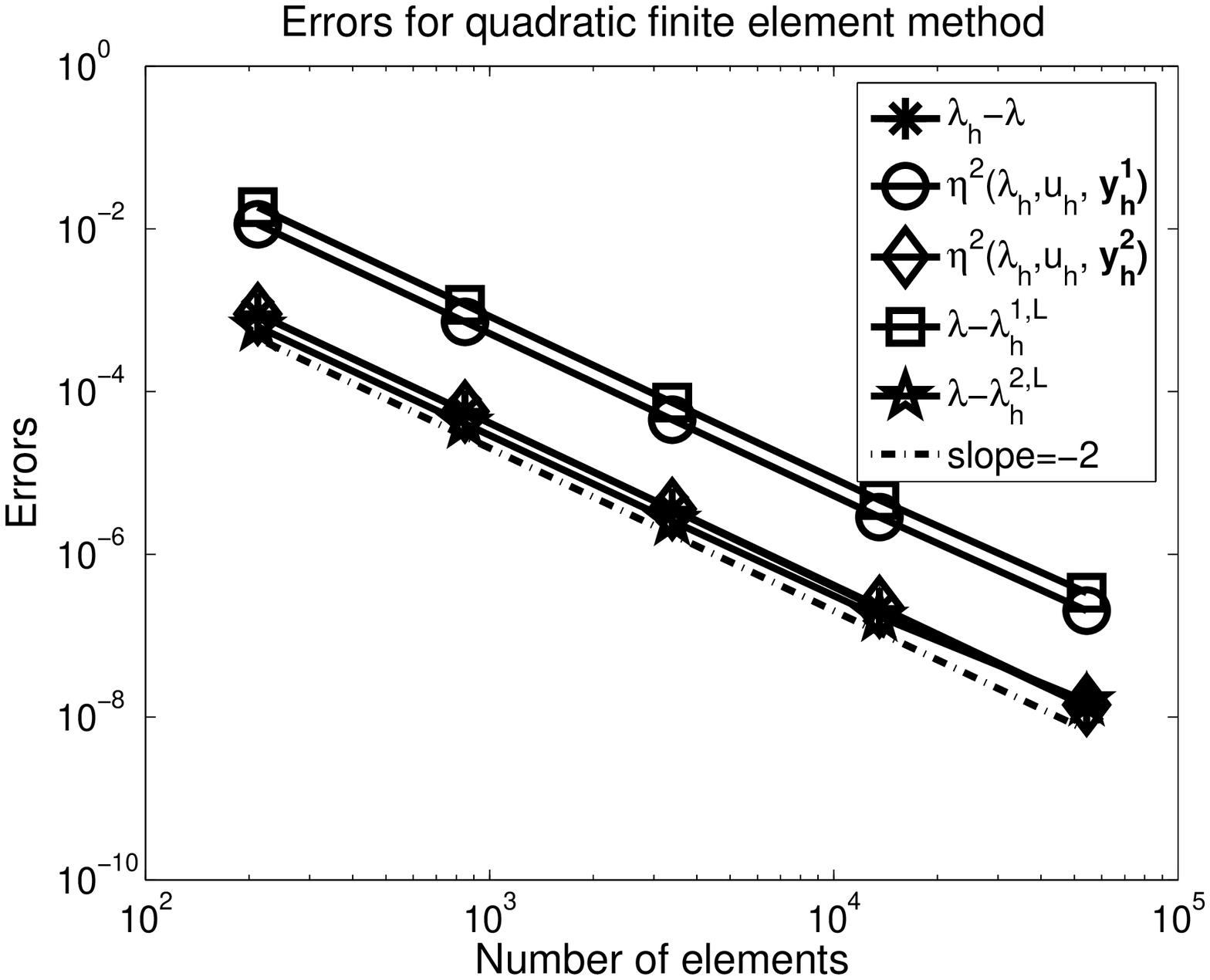}
\caption{\small\texttt The errors for the unit square domain when the eigenvalue problem is solved
by the quadratic finite element method, where $\eta(\lambda_h,u_h,\mathbf y_h^1)$ and
$\eta(\lambda_h,u_h,\mathbf y_h^2)$ denote the a posteriori error estimates $\eta(\lambda_{1,h},u_{1,h},\mathbf y_h^*)$
when the dual problem is solved by $\mathbf W_h^1$ and $\mathbf W_h^2$, respectively, and $\lambda_h^{1,L}$ and
$\lambda_h^{2,L}$ denote the guaranteed lower bounds of the first eigenvalue $\lambda_1$ when the dual problem is
solved by $\mathbf W_h^1$ and $\mathbf W_h^2$, respectively}
\label{Exam_1_P_2_RT1_RT2}
\end{figure}

In this section, we also check the efficiency of the error estimates $\eta^2(\lambda_{i,h},u_{i,h},\mathbf y_h^*)$
($i=2,3$) for the second and third eigenvalues. Tables \ref{Example_1_Table_1} and \ref{Example_1_Table_2} show
the corresponding numerical results. In Table \ref{Example_1_Table_1}, we solve the eigenvalue
problem (\ref{Weak_Eigenvalue_Discrete}) by the linear finite element method and the
dual problem (\ref{Dual_Problem_Discrete}) with the finite element space $\mathbf W_h^0$
and $\mathbf W_h^1$, respectively.
In Table \ref{Example_1_Table_2}, the eigenvalue problem (\ref{Weak_Eigenvalue_Discrete}) is solved
by the quadratic finite element method and we solve  the dual problem (\ref{Dual_Problem_Discrete})
with the finite element space $\mathbf W_h^1$ and $\mathbf W_h^2$, respectively.
\begin{table}[ht]
\centering
\caption{\footnotesize\texttt The errors for the unit square domain when the eigenvalue problem is solved
by the linear finite element method, where $\eta(\lambda_{i,h},u_{i,h},\mathbf y_h^0)$ ($i=2,3$) and
$\eta(\lambda_{i,h},u_{i,h},\mathbf y_h^1)$ denote the a posteriori error estimates $\eta(\lambda_{i,h},u_{i,h},\mathbf y_h^*)$
when the dual problem is solved by $\mathbf W_h^0$ and $\mathbf W_h^1$, respectively.}\label{Example_1_Table_1}
\begin{tabular}{||c|c|c|c||}
\hline
\footnotesize{Number of elements} & $\lambda_{2,h}-\lambda_2$& $\eta^2(\lambda_{2,h},u_{2,h},\mathbf y_h^0)$  & $\eta^2(\lambda_{2,h},u_{2,h},\mathbf y_h^1)$ \\
\hline
208    &     1.9304e+00 & 7.2113e+01 & 1.9875e+00  \\
\hline
832    &    4.8497e-01 & 1.6651e+01 & 4.8866e-01  \\
\hline
3328   &    1.2164e-01 & 4.0794e+00 & 1.2188e-01  \\
\hline
13312  &    3.0450e-02 & 1.0147e+00 & 3.0469e-02  \\
\hline
53248  &    7.6161e-03 & 2.5337e-01 & 7.6182e-03  \\
\hline
212992 &   1.9043e-03 & 6.3322e-02 & 1.9047e-03  \\
\hline
\footnotesize{Number of elements} &  $\lambda_{3,h}-\lambda_3$& $\eta^2(\lambda_{3,h},u_{3,h},\mathbf y_h^0)$  & $\eta^2(\lambda_{3,h},u_{3,h},\mathbf y_h^1)$ \\
\hline
208    & 1.9386e+00 & 7.0685e+01  &1.9968e+00\\
\hline
832    & 4.8728e-01 & 1.6198e+01  &4.9098e-01\\
\hline
3328   & 1.2227e-01 & 3.9655e+00  &1.2252e-01\\
\hline
13312  & 3.0615e-02 & 9.8627e-01  &3.0634e-02\\
\hline
53248  & 7.6578e-03 & 2.4625e-01  &7.6599e-03\\
\hline
212992 & 1.9148e-03 & 6.1543e-02  &1.9151e-03\\
\hline
\end{tabular}
\end{table}

The numerical results in Tables \ref{Example_1_Table_1} and \ref{Example_1_Table_2}
show that $\eta^2(\lambda_{i,h},u_{i,h},\mathbf y_h^*)$ ($i=2,3$) is a very efficient error estimator for the
eigenvalue approximation $\lambda_{i,h}$ when the error of the dual problem is small compared to the
error of the primitive problem. This phenomena is in agreement with Theorem \ref{Efficiency_Eigenvalue_Theorem},
Corollary \ref{Lower_Bound_Eigen_Corollary} and Remark \ref{Lower_Bound_Eigen_Remark}.

\begin{table}[ht]
\centering
\caption{\footnotesize\texttt The errors for the unit square domain when the eigenvalue problem is solved
by the quadratic finite element method, where $\eta(\lambda_{i,h},u_{i,h},\mathbf y_h^0)$ ($i=2,3$) and
$\eta(\lambda_{i,h},u_{i,h},\mathbf y_h^1)$ denote the a posteriori error estimates
$\eta(\lambda_{i,h},u_{i,h},\mathbf y_h^*)$ when the dual problem is solved by $\mathbf W_h^0$ and $\mathbf W_h^1$, respectively.}\label{Example_1_Table_2}
\begin{tabular}{||c|c|c|c||}
\hline
\footnotesize{Number of elements} & $\lambda_{2,h}-\lambda_2$& $\eta^2(\lambda_{2,h},u_{2,h},\mathbf y_h^0)$  & $\eta^2(\lambda_{2,h},u_{2,h},\mathbf y_h^1)$ \\
\hline
208    &    1.3955e-02 &   4.6633e-01 &   1.3818e-02   \\
\hline
832    &    9.0239e-04 &   2.9777e-02 &   9.0013e-04   \\
\hline
3328   &    5.7163e-05 &   1.8719e-03 &   5.7128e-05   \\
\hline
13312  &    3.5934e-06 &   1.1718e-04 &   3.5928e-06   \\
\hline
53248  &    2.2519e-07 &   7.3268e-06 &   2.2519e-07   \\
\hline
\footnotesize{Number of elements} &  $\lambda_{3,h}-\lambda_3$& $\eta^2(\lambda_{3,h},u_{3,h},\mathbf y_h^0)$  & $\eta^2(\lambda_{3,h},u_{3,h},\mathbf y_h^1)$ \\
\hline
208    &       1.4340e-02 &  4.6548e-01 &  1.4193e-02  \\
\hline
832    &       9.2527e-04 &  2.9950e-02 &  9.2287e-04  \\
\hline
3328   &       5.8616e-05 &  1.8858e-03 &  5.8578e-05  \\
\hline
13312  &       3.6855e-06 &  1.1809e-04 &  3.6849e-06  \\
\hline
53248  &       2.3101e-07 &  7.3849e-06 &  2.3100e-07  \\
\hline
\end{tabular}
\end{table}

\subsection{Eigenvalue problem on L-shape domain}
In the second example, we solve the eigenvalue problem (\ref{weak_eigenvalue_problem})
on the L-shape domain $\Omega=(-1,1)\times (-1,1)/[0,1)\times (-1,0]$.
 Since $\Omega$ has a re-entrant corner, the singularity of the first eigenfunction
is expected. The convergence order for the eigenvalue approximation is
less than $2$ by the linear finite element method which is the order predicted by the
theory for regular eigenfunctions.
We investigate the numerical results for the first eigenvalue. Since the exact
eigenvalue is not known, we choose an adequately accurate approximation
$\lambda_1 = 10.6397238440219$ obtained by the extrapolation method \cite{LinLin}
as the exact first eigenvalue for the numerical tests.
In order to treat the singularity of the eigenfunction, we solve the eigenvalue problem
(\ref{weak_eigenvalue_problem}) by the adaptive finite element method (cf. \cite{BrennerScott}).
For simplicity, we set $\lambda:=\lambda_1$, $u:=u_1$, $\lambda_h:=\lambda_{1,h}$ and $u_h:=u_{1,h}$
in this subsection.

We present this example to validate the results in this paper also hold on the adaptive meshes.
In order to use the adaptive finite element method, we define the a posteriori error estimator as follows:
Define the element residual $\mathcal{R}_K(\lambda_h,u_h)$ and the jump residual $\mathcal{J}_E(u_h)$ as
follows:
\begin{eqnarray}
\mathcal{R}_K(\lambda_h,u_h)&:=&\lambda_hu_h
 +\Delta u_h- u_h \ \ \text{in}\ K\in \mathcal{T}_h,\\
\mathcal{J}_E(u_h)&:=&-\nabla u_h^+ \cdot \nu^+- \nabla u_h^- \cdot \nu^-
:=[[\nabla u_h]]_E\cdot\nu_E\ \ \ \text{on}\  E\in \mathcal{E}_h,
\end{eqnarray}
where $E$ is the common side of elements $K^{+}$ and $K^-$ with outward
normals $\nu^+$ and $\nu^-$, $\nu_E=\nu^-$.

For each element $K\in \mathcal{T}_h$, we define the local error
indicator $\eta_h(\lambda_h,u_h,K)$ by
\begin{eqnarray}\label{eta_definition}
\eta_h^2(\lambda_h,u_h,K):=h_T^2\|\mathcal{R}_K(\lambda_h,u_h)\|_{0,K}^2
+\sum\limits_{E\in \mathcal{E}_h,E\subset
 \partial K}h_E\|\mathcal{J}_E(u_h)\|^2_{0,E}.
\end{eqnarray}
Then we define the global a posteriori error estimator
$\eta_{\rm ad}(\lambda_h,u_h)$ by
\begin{eqnarray}
\eta_{\rm ad}(\lambda_h,u_h):=
\left(\sum_{K\in \mathcal{T}_h}\eta_h^2(\lambda_h,u_h,K)\right)^{1/2}.
\end{eqnarray}

We solve the eigenvalue
problem (\ref{Weak_Eigenvalue_Discrete})
by the linear conforming finite element method and solve the dual problem (\ref{Dual_Problem_Discrete})
in the finite element space $\mathbf W_h^0$ and $\mathbf W_h^1$, respectively.
Figure \ref{Mesh_AFEM_Exam_2} (left) shows the corresponding adaptive mesh.
The corresponding numerical results are presented in Figure \ref{Exam_2_P_1_RT0_RT1}
which shows that the a posteriori error estimate $\eta(\lambda_h,u_h,\mathbf y_h^*)$ is also efficient
even on the adaptive meshes when we solve the dual problem by  $\mathbf W_h^1$.
Figure \ref{Exam_2_P_1_RT0_RT1} also shows the validation of the guaranteed upper bound
 $\eta_h^U(\lambda_h,u_h,\mathbf y_h^*)$ for the error $\|u-u_h\|_a$ and the eigenvalue approximation $\lambda_h^L$
 is really a guaranteed lower bound of the first eigenvalue despite the way to
 solve the dual problem by $\mathbf W_h^0$ or $\mathbf W_h^1$.

\begin{figure}[ht]
\centering
\includegraphics[width=5cm,height=5cm]{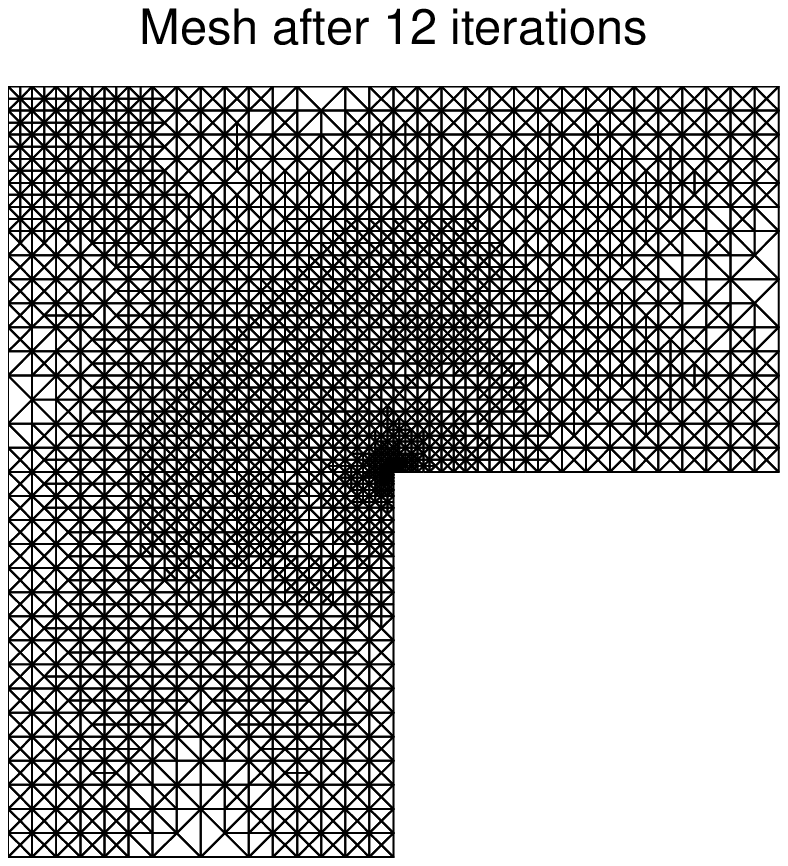}
\includegraphics[width=5cm,height=5cm]{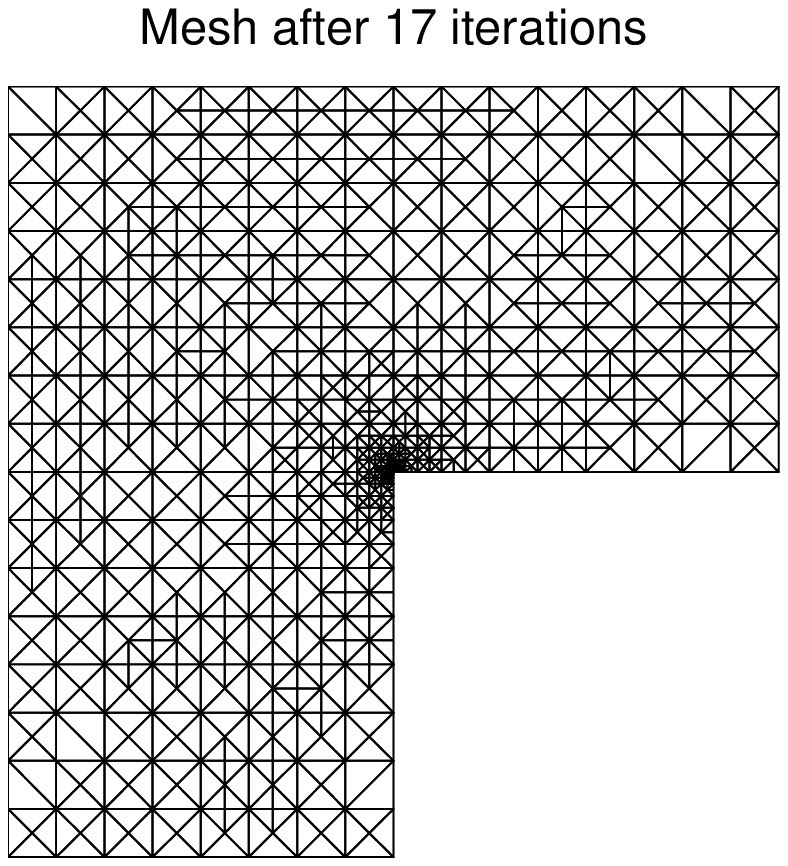}
\caption{The triangulations after adaptive iterations for
L-shape domain by the linear element (left) and the quadratic element (right)}\label{Mesh_AFEM_Exam_2}
\end{figure}

\begin{figure}[ht]
\centering
\includegraphics[width=6.5cm,height=6cm]{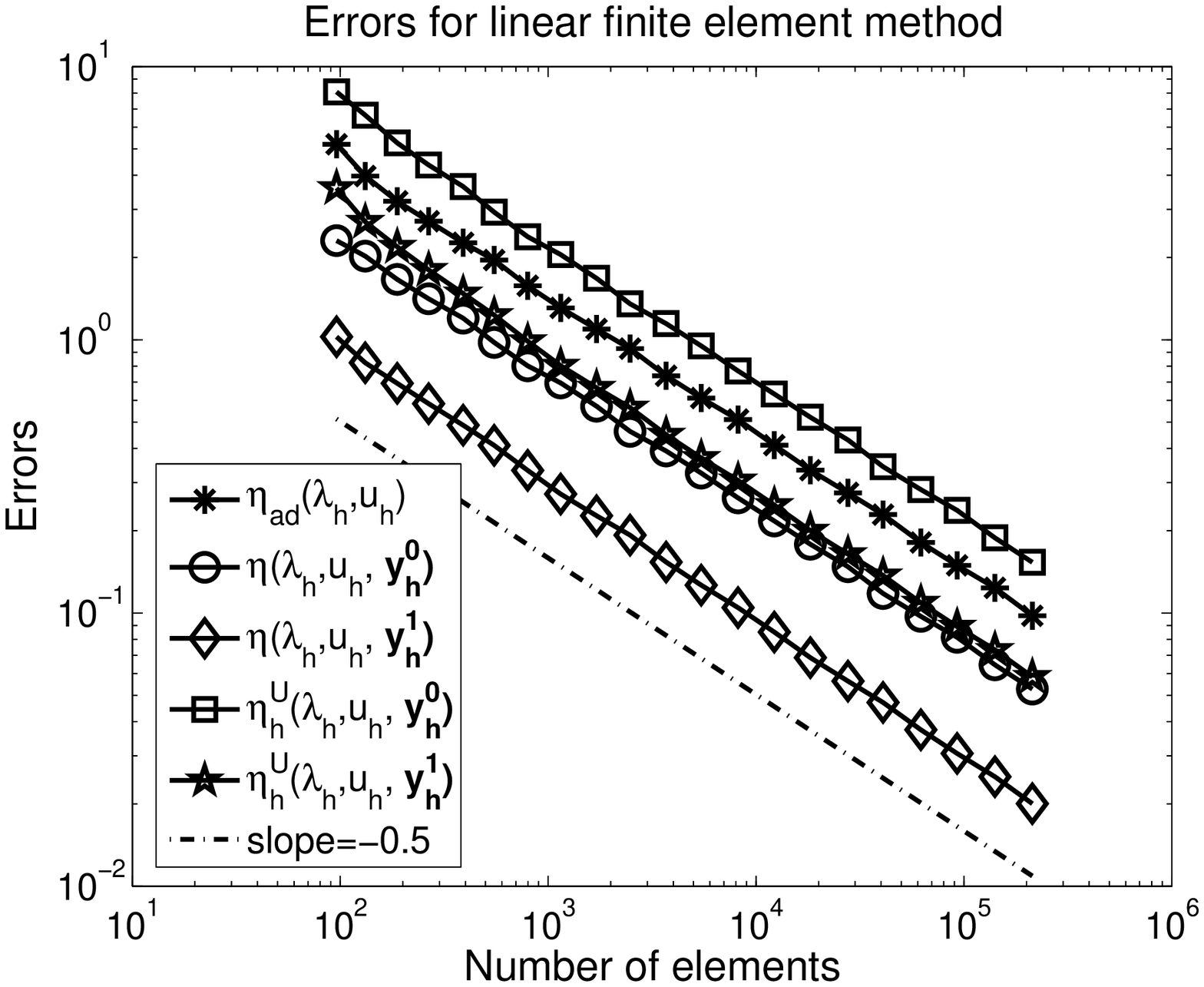}
\includegraphics[width=6.5cm,height=6cm]{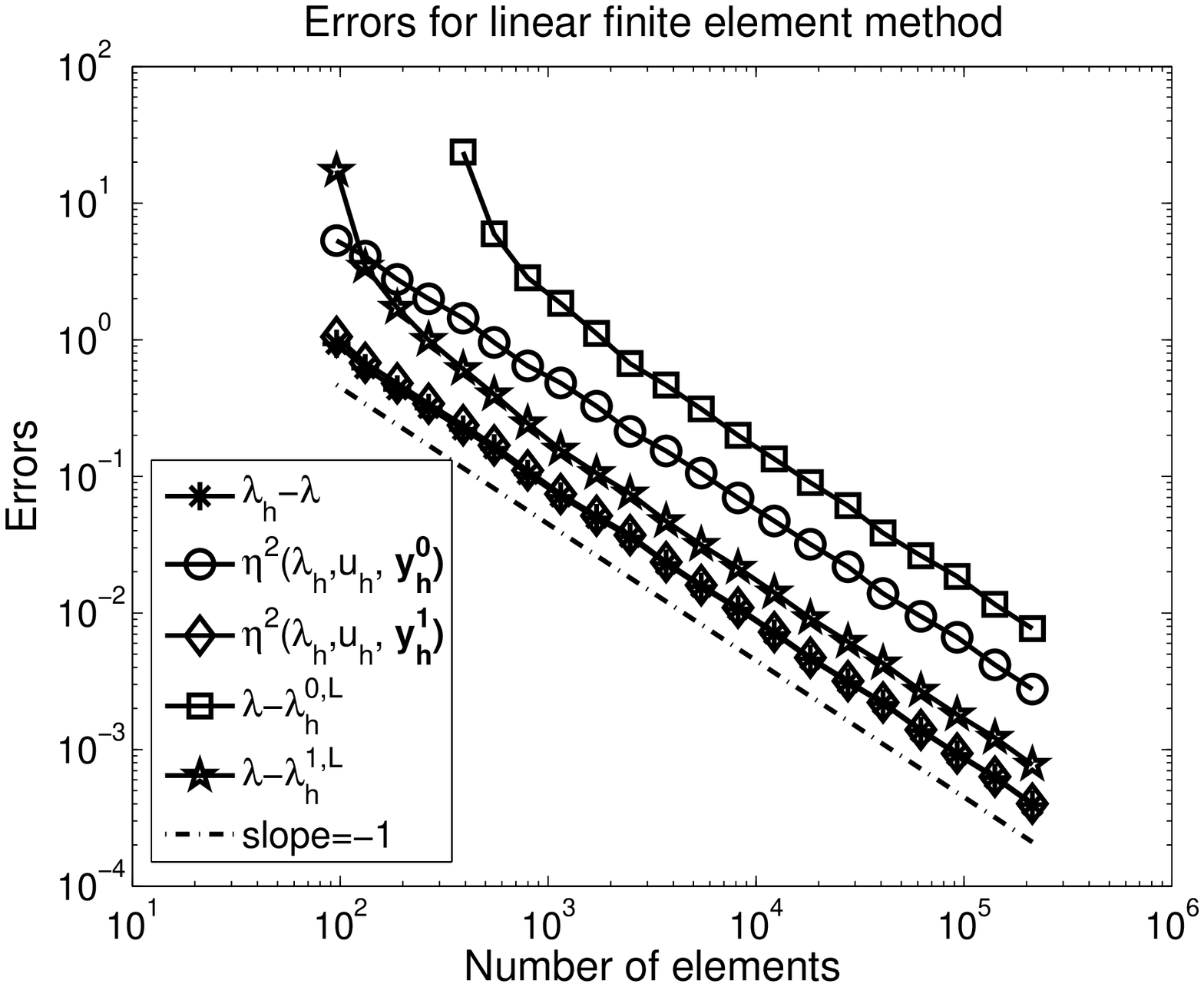}
\caption{\small\texttt The errors for the L-shape domain when the eigenvalue problem is solved
by the linear finite element method, where $\eta(\lambda_h,u_h,\mathbf y_h^0)$ and
$\eta(\lambda_h,u_h,\mathbf y_h^1)$ denote the a posteriori error estimates $\eta(\lambda_h,u_h,\mathbf y_h^*)$
when the dual problem is solved by $\mathbf W_h^0$ and $\mathbf W_h^1$, respectively, and $\lambda_h^{0,L}$ and
$\lambda_h^{1,L}$ denote the guaranteed lower bounds of the first eigenvalue when the dual problem is
solved by $\mathbf W_h^0$ and $\mathbf W_h^1$, respectively}
\label{Exam_2_P_1_RT0_RT1}
\end{figure}

In this example, we also solve the eigenvalue problem (\ref{Weak_Eigenvalue_Discrete}) by the quadratic
finite element method and the dual problem (\ref{Dual_Problem_Discrete})
with the finite element space $\mathbf W_h^1$ and $\mathbf W_h^2$, respectively. The corresponding adaptive mesh
is presented in Figure \ref{Mesh_AFEM_Exam_2} (right).
Figure \ref{Exam_2_P_2_RT1_RT2} shows the corresponding numerical results. From Figure \ref{Exam_2_P_2_RT1_RT2},
we can find that the a posteriori error estimate $\eta(\lambda_h,u_h,\mathbf y_h^*)$ is efficient
when we solve the dual problem by  $\mathbf W_h^2$. Figure \ref{Exam_2_P_2_RT1_RT2} also shows
$\eta_h^U(\lambda_h,u_h,\mathbf y_h^*)$ is really the guaranteed upper bound of the error $\|u-u_h\|_a$
and the eigenvalue approximation $\lambda_h^L$ is also really a guaranteed lower bound of the first eigenvalue.

\begin{figure}[ht]
\centering
\includegraphics[width=6.5cm,height=6cm]{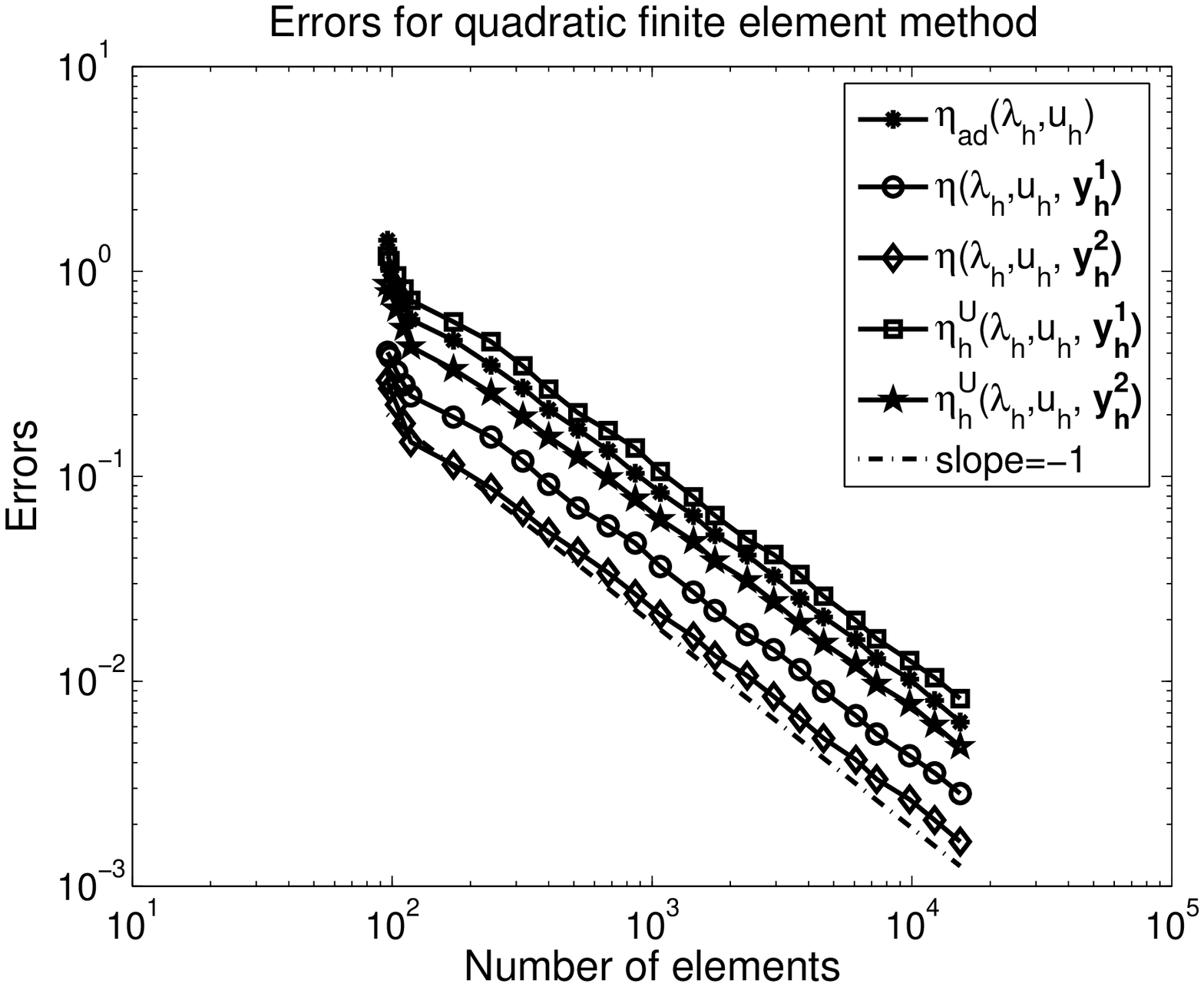}
\includegraphics[width=6.5cm,height=6cm]{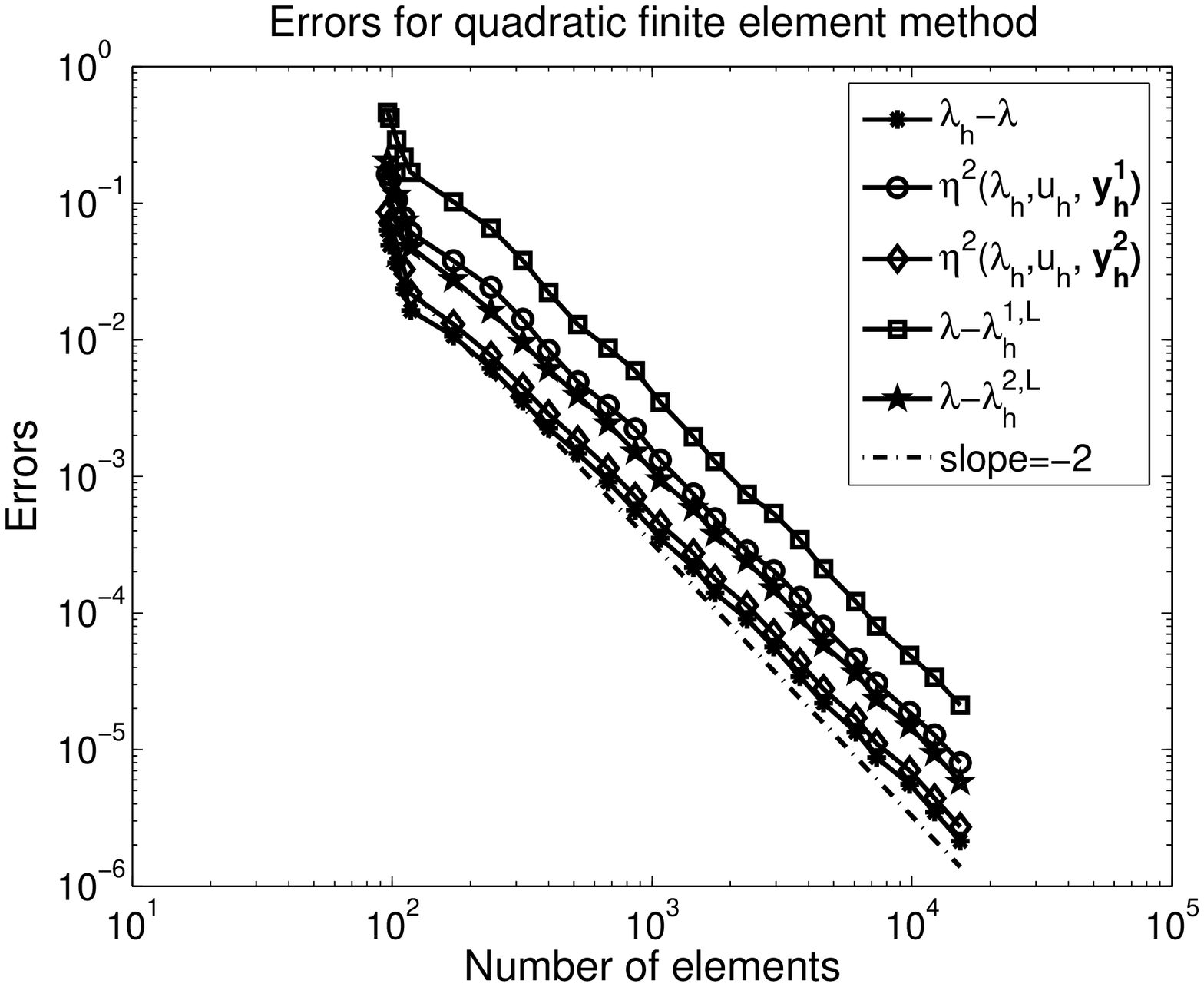}
\caption{\small\texttt The errors for the L shape domain when the eigenvalue problem is solved
by the quadratic finite element method, where $\eta(\lambda_h,u_h,\mathbf y_h^1)$ and
$\eta(\lambda_h,u_h,\mathbf y_h^2)$ denote the a posteriori error estimates $\eta(\lambda_h,u_h,\mathbf y_h^*)$
when the dual problem is solved by $\mathbf W_h^1$ and $\mathbf W_h^2$, respectively, and $\lambda_h^{1,L}$ and
$\lambda_h^{2,L}$ denote the guaranteed lower bounds of the first eigenvalue when the dual problem is
solved by $\mathbf W_h^1$ and $\mathbf W_h^2$, respectively}
\label{Exam_2_P_2_RT1_RT2}
\end{figure}

\section{Concluding remarks}
In this paper, we give a computable error estimate for the eigenpair approximation by the
general conforming  finite element methods on general meshes. Furthermore, the guaranteed
upper bound of the error estimate for the first eigenfunction approximation and the lower
bound of the first eigenvalue can be obtained by the computable error estimate and a
lower bound of the second eigenvalue. If the eigenpair approximations are obtained by solving
the discrete eigenvalue problem, the computable error estimates are asymptotically exact and
we can also give asymptotically lower bounds for the general eigenvalues. Some numerical examples
are provided to demonstrate the validation of the guaranteed upper and lower bounds for the
general conforming finite element methods on the general meshes (quasi-uniform and regular types
 \cite{BrennerScott,Ciarlet}).
The method here can be extended to other eigenvalue problems such as Steklov, Stokes and
other similar types \cite{LinXie,SebestovaVejchodsky}.
Especially, we would like to say that the computable error estimate can be extended to the nonlinear
 eigenvalue problems
which are produced from the complicated linear eigenvalue problems. Furthermore, the method in this
paper can be used to check the modeling and discretization errors for the models
(nonlinear eigenvalue problems) in the density functional theory comes from the linear Schr\"{o}dinger
equation \cite{KohnSham,Martin}. These will be our future work.

\section*{Acknowledge}
We would like to thank Tom\'{a}\v{s} Vejchodsk\'{y} for his kindly discussion!



\begin{thebibliography}{17}

\bibitem{Adams}
R. Adams, {\em Sobolev Spaces}, Academic Press, New York, 1975.

\bibitem{AinsworthOden}
M. Ainsworth and J. Oden, {\em A posteriori error estimation in finite element
analysis}, Pure and Applied Mathematics (New York), Wiley-Interscience [John
Wiley \& Sons, New York, 2000.

\bibitem{ArmentanoDuran}
M. G. Armentano and R. G. Dur\'{a}n, {\em Asymptotic lower bounds for eigenvalues by nonconforming finite element
method},  Electron. Trans. Numer. Anal., 17 (2004), 93-101.

\bibitem{BabuskaOsborn_1989}
I. Babu\v{s}ka and J. Osborn, {\em Finite element-Galerkin
approximation of the eigenvalues and eigenvectors of selfadjoint
problems}, Math. Comp. 52 (1989), 275-297.

\bibitem{BabuskaOsborn_Book}
I. Babu\v{s}ka and J. Osborn, {\em Eigenvalue Problems}, In Handbook of
Numerical Analysis, Vol. II, (Eds. P. Lions and P. Ciarlet),
Finite Element Methods (Part 1), North-Holland, Amsterdam, 641-787,
1991.

\bibitem{BabuskaRheinboldt_1}
I. Babu\v{s}ka and W. Rheinboldt, {\em Error estimates for adaptive finite ele-
ment computations}, SIAM J. Numer. Anal., 15 (1978), 736-754.

\bibitem{BabuskaRheinboldt_2}
I. Babu\v{s}ka and W. Rheinboldt, {\em A-posteriori error estimates for the finite
element method}, Int. J. Numer. Methods Eng., 12 (1978), 1597-1615.

\bibitem{BrennerScott}
S. Brenner and L. Scott, {\em The Mathematical Theory of Finite Element
Methods}, New York: Springer-Verlag, 1994.

\bibitem{BrezziFortin}
F. Brezzi and M. Fortin, {\em Mixed and Hybrid Finite Element Methods}, New York: Springer-Verlag,  1991.

\bibitem{CarstensenGallistl}
C. Carstensen and D. Gallistl, {\em Guaranteed lower eigenvalue bounds for the biharmonic equation},
Numer. Math., 126 (2014),  33-51.

\bibitem{CarstensenGedicke}
C. Carstensen and J. Gedicke, {\em Guaranteed lower bounds for eigenvalues},
Math. Comp., 83(290)  (2014), 2605-2629.

\bibitem{Chatelin}
F. Chatelin, {\em Spectral Approximation of Linear Operators}, Academic
Press Inc, New York, 1983.

\bibitem{Ciarlet}
P. Ciarlet, {\em The Finite Element Method for Elliptic Problem},
North-holland Amsterdam, 1978.


\bibitem{HaslingerHlavacek}
J. Haslinger and I. Hlav\'{a}\v{c}ek, {\em Convergence of a finite element method based on the dual
variational formulation}, Apl. Mat., 21 (1976), 43-65.

\bibitem{HuHuangLin}
J. Hu, Y. Huang and Q. Lin, {\em The lower bounds for eigenvalues of elliptic
operators by nonconforming finite element methods}, J. Sci. Comput., 61(1) (2014), 196-221.

\bibitem{KohnSham}
W. Kohn and L. Sham, {\em Self-consistent equations including exchange and correlation effects},
Phys. Rev. A, 140 (1965), 4743-4754.

\bibitem{LinLuoXie_lowerbound}
Q. Lin, F. Luo and H. Xie, {\em A posterior error estimator and lower bound of a
nonconforming finite element method}, J. Comput. App. Math., 265 (2014), 243-254.

\bibitem{LinXie_lowerbound}
Q. Lin and H. Xie, {\em The asymptotic lower bounds of eigenvalue problems by nonconforming
finite element methods}, Math. in Practice and Theory (in Chinese), 42(11) (2012), 219-226.


\bibitem{LinLin}
Q. Lin  and J. Lin, {\em Finite Element Methods: Accuracy and Inprovement},
Science Press, Beijing, 2006.

\bibitem{LinXie}
Q. Lin and H. Xie, {\em Recent results on lower bounds of eigenvalue problems by
nonconforming finite element methods},
 Inverse Problems and Imaging, 7(3), 2013, 795-811.

\bibitem{LinXieLuoLiYang}
Q. Lin, H. Xie, F. Luo, Y. Li, and Y. Yang, {\em Stokes eigenvalue approximation from below
with nonconforming mixed finite element methods}, Math. in Practice and Theory (in Chinese),
19 (2010), 157-168.

\bibitem{Liu}
X. Liu, {\em A framework of verified eigenvalue bounds for self-adjoint differential operators},
Appl. Math. Comput., 267 (2015), 341-355.

\bibitem{LiuOishi}
X. Liu and S. Oishi, {\em Verifed eigenvalue evaluation for the Laplacian over
polygonal domains of arbitrary shape}, SIAM J. Numer. Anal., 51(3) (2013), 1634-1654.

\bibitem{LuoLinXie}
F. Luo, Q. Lin and H. Xie, {\em Computing the lower and upper bounds of Laplace eigenvalue problem:
by combining conforming and nonconforming finite element methods},  Sci. China Math., 55(5) (2012), 1069-1082.

\bibitem{Martin}
R. Martin, {\em Electronic Structure: Basic Theory and Practical Methods}, Cambridge
University Press, London, 2004.

\bibitem{NeittaanmakiRepin}
P. Neittaanm\"{a}ki and S. Repin, {\em Reliable methods for computer simulation,
Error control and a posteriori estimates}, vol. 33 of Studies in Mathematics and
its Applications, Elsevier Science B. V., Amsterdam, 2004.

\bibitem{Repin}
S. Repin, {\em A posteriori estimates for partial differential equations}, vol. 4 of Radon
Series on Computational and Applied Mathematics, Walter de Gruyter GmbH
\& Co. KG, Berlin, 2008.


\bibitem{SebestovaVejchodsky}
I. \v{S}ebestov\'{a} and T. Vejchodsk\'{y}, {\em Two-sided bounds for eigenvalues of differential
operators with applications to Friedrichs', Poincar\'{e}, trace, and similar constants},
SIAM J. Numer. Anal., 52(1) (2014), 308-329.

\bibitem{Vejchodsky_1}
T. Vejchodsk\'{y}, {\em Complementarity based a posteriori error estimates and their properties},
Math. Comput. Simulation, 82 (2012), 2033-2046.

\bibitem{Vejchodsky_2}
T. Vejchodsk\'{y}, {\em Computing upper bounds on Friedrichs' constant}, in Applications of Mathematics
2012, J. Brandts, J. Chleboun, S. Korotov, K. Segeth, J. \v{S}\'{i}stek, and T. Vejchodsk\'{y},
eds., Institute of Mathematics, ASCR, Prague, 2012, pp. 278-289.

\bibitem{Verfurth}
R. Verf\"{u}rth, {\em A review of a posteriori error estimation and adaptive mesh-refinement
 techniques}, Wiley-Teubner, Chichester/Stuttgart, 1996.

\bibitem{Xie_JCP}
H. Xie, {\em A multigrid method for eigenvalue problem}, J. Comput. Phys., 274 (2014),  550-561.

\bibitem{Xie_IMA}
H. Xie, {\em A type of multilevel method for the Steklov eigenvalue problem},
IMA J. Numer. Anal., 34 (2014), 592-608.

\bibitem{YangZhangLin}
Y. Yang, Z. Zhang and F. Lin, {\em Eigenvalue approximation from below using non-forming finite elements},
Sci. China. Math., 53 (2010), 137-150.

\bibitem{ZhangYangChen}
Z. Zhang, Y. Yang and Z. Chen, {\em Eigenvalue approximation from below by Wilson's element},
Chinese J. Numer. Math. Appl., 29 (2007), 81-84.


\end{thebibliography}
\end{document}